\documentclass{amsart}
\usepackage[T1]{fontenc}
\usepackage{txfonts}
\usepackage{geometry}
\usepackage[all]{xy}
\usepackage{enumerate}
\usepackage{mdwlist}
\usepackage{amsmath}
\usepackage{amssymb}
\usepackage{amscd}
\usepackage{fixltx2e}
\usepackage[hidelinks]{hyperref}

\numberwithin{equation}{subsection}
\newtheorem{lemma}[equation]{Lemma}
\newtheorem{proposition}[equation]{Proposition}
\newtheorem{theorem}[equation]{Theorem}
\newtheorem{corollary}[equation]{Corollary}
\newtheorem{thm}{Theorem}
\theoremstyle{remark}
\newtheorem*{remark}{Remark}
\newtheorem{cx}[equation]{Counterexample}

\DeclareMathOperator{\Spec}{Spec}
\DeclareMathOperator{\Proj}{Proj}
\DeclareMathOperator{\Sym}{Sym}
\DeclareMathOperator{\gp}{gp}

\DeclareMathOperator{\cha}{char}
\DeclareMathOperator{\td}{tr.deg}
\DeclareMathOperator{\spe}{sp}
\DeclareMathOperator{\cok}{cok}

\DeclareMathOperator{\dlog}{dlog}
\DeclareMathOperator{\Grass}{Grass}
\DeclareMathOperator{\sm}{sm}
\DeclareMathOperator{\im}{im}
\DeclareMathOperator{\an}{an}

\DeclareMathOperator{\pr}{pr}
\DeclareMathOperator{\rk}{rk}
\DeclareMathOperator{\et}{\acute{e}t}
\DeclareMathOperator{\Fet}{f\acute{e}t}

\begin{document}

\title{Bertini theorems for singular schemes and nearby cycles in mixed characteristic}

\author{R\'emi Lodh}
\email{remi.shankar@gmail.com}
\begin{abstract}
We first study hyperplane sections of some singular schemes over a field. We prove a Bertini theorem for the log smoothness of generic hyperplane sections of a large class of log smooth schemes over a log point. We also give an abstract generalization of the usual Bertini theorem for smoothness. We then apply this result to the study of hyperplane sections of regular schemes over a complete discrete valuation ring of characteristic zero with algebraically closed residue field. Lastly, give an exposition of a result of Faltings which states that one can compute the $p$-adic nearby cycles of smooth schemes over a discrete valuation ring of mixed characteristic as Galois cohomology.
\end{abstract}
\maketitle

\section*{Introduction}
The purpose of this paper is to study hyperplane sections of certain schemes, including log smooth schemes. We first study the case of schemes over a field $k$ and prove two variants of Bertini's theorem for smoothness. The first one is a Bertini theorem for log smoothness. To state it, let be a monoid $Q$ and a map $Q\to k$ such that $(k,Q)$ is a \emph{log point} (\ref{setup2}). Assume $X\subset\mathbb{P}^n_k$ is a subscheme, $H\subset\mathbb{P}^n_k$ a hyperplane, $X\cap H$ the scheme-theoretic intersection, and $i:X\cap H\subset X$ the inclusion. Let $M$ be a log structure on $X$ and $f:(X,M)\to (k,Q)$ a morphism of log schemes.

\begin{thm}[Bertini theorem for log smoothness (\ref{bertini6})]
Assume $f$ is a log smooth morphism. If $f$ is a \emph{sharp} morphism of log schemes, then for $H$ generic $(X\cap H, i^*M)$ is log smooth over $(k,Q)$.
\end{thm}

The notion of sharp morphism of log schemes is defined in \ref{sharp}. Suffices to say here that sharpness is a quite natural condition which is mostly pertinent to positive characteristic (cf. \ref{tame}). The theorem is optimal in the sense that we show by example that the sharpness assumption cannot be dropped in general (\ref{cx}).

This Bertini theorem applies in particular to all varieties with toric singularities as well as singular schemes which are locally isomorphic to a normal crossings divisor in a smooth variety. In fact, the latter case was previously shown by Jannsen and Saito \cite{jannsensaito} by reducing to the Bertini theorem for classical smoothness. By contrast, our proof is a direct logarithmic adaptation of the transcendence degree argument given by Jouanolou in \cite{bertini}.

The second variant of Bertini's theorem, hinted at in \cite[2.1]{fa2}, does not make any assumption on the singularities of $X$, but rather assumes the existence of a morphism of finite type $m:\mathcal{X}\to X$ and a vector bundle quotient $\mathcal{E}$ of $m^*\Omega^1_{X/k}$.

\begin{thm}[Bertini theorem for abstract smoothness (\ref{bertini7})]\label{bertabstract}
Let $\mathcal{N}$ denote the conormal sheaf of $X\cap H$ in $X$, and consider the natural morphism
\[ \phi:m^*\mathcal{N}\to i^*\mathcal{E} \]
of coherent sheaves on $\mathcal{X}\times_X(X\cap H)$. If $\dim\mathcal{X}\leq\rk\mathcal{E}$, then for $H$ generic the cokernel of $\phi$ is a vector bundle of rank $\rk\mathcal{E}-1$ on $\mathcal{X}\times_X(X\cap H)$.
\end{thm}

In the case $X$ is smooth and $\mathcal{X}=X$ one recovers the classical Bertini theorem for smoothness; in fact, the proof of the latter given in \cite{bertini} applies mutatis mutandis.

We then study hyperplane sections of singular schemes with the additional constraint that the hyperplanes all pass through a given point. In this case simple examples show that, unlike the classical smooth case, for $X$ log smooth one cannot guarantee that a generic hyperplane section through $P\in X$ is log smooth at the point $P$ (take $X$ to be the union of coordinate axes in affine space and $P$ the origin). The interesting case is when $X$ is defined over a complete discrete valuation ring $V$ with singular special fibre. Assume the residue field $k$ of $V$ is algebraically closed and let $K$ denote the fraction field of $V$.

\begin{thm}[\ref{maingoodnhbprop}]\label{goodetalemap}
Assume $\cha(K)=0$. Let $X$ be a flat projective $V$-scheme with geometrically reduced generic fibre of dimension $d$ everywhere. Fix a closed point $P\in X_k:=X\otimes_Vk$. If $X$ is regular at $P$, then, up to modifying the projective embedding of $X$, for generic $V$-hyperplanes $H_1,...,H_d$ passing through $P$ the $V$-scheme $X\cap H_1\cap\cdots\cap H_d$ has \'etale generic fibre.
\end{thm}

The basic idea of the proof is to find a nice model of $X_K:=X\otimes_VK$ (in fact the blow up of $X$ at $P$) which has a convenient projective embedding for taking hyperplane sections.

Finally, we use the results on hyperplane sections to construct $K(\Pi,1)$-neighbourhoods of smooth $V$-schemes following Faltings \cite{fa2}, and deduce that the $p$-adic nearby cycles can be computed locally as the Galois cohomology of the geometric generic fibre. For another exposition of this result, see Olsson \cite{olsson}.

\subsection*{Acknowledgement}
In a previous version of this paper some claims were incorrect. I am indebted to Piotr Achinger for providing counterexamples to those claims.

\subsection*{Overview of the paper}
In \S1 we define the notion of a sharp morphism of log schemes and give examples. Then we recall some facts about log smooth morphisms.

In \S2 we set up and prove affine versions of the Bertini theorems. We first show the Bertini theorem for log smoothness \ref{logbertini}. The idea of the proof is as follows. The argument of \cite{bertini} is that if the generic hyperplane section of a smooth scheme $X$ were singular at a point $x$, then the resulting $k(x)$-linear relation between differentials in $\Omega^1_{k(x)}$ would force the transcendence degrees to be too small. Our proof reduces to this very same argument. The key is that for a point $x$ on a \emph{sharp} log scheme $X$ over a log point, the usual differentials $\Omega^1_{k(x)}$ inject into the logarithmic differentials $\omega^1_{k(x)}$ (\ref{satz2}), therefore any relation between regular differentials in the latter is in fact a relation in the former. Then in \ref{cx} we construct a counterexample which shows that the theorem is false in general for log smooth schemes which are not sharp. Afterwards we prove the Bertini theorem for abstract smoothness by applying the argument of \cite{bertini}.

In \S3 we show that the Bertini theorems in the quasi-projective case follow from the affine case already shown (\ref{bertini6}, \ref{bertini7}). Here we also follow the method sketched in \cite{bertini}.

In \S4, we study the case of schemes over a complete discrete valuation ring $V$. After some preliminaries in \ref{special}, we study the case of a generic hyperplane through a given rational point $P\in\mathbb{P}^n(k)$. For classical smoothness, up to modifying the projective embedding there is a Bertini theorem (cf. \cite[Exp. XI]{sga4.3}); the underlying idea is that in a suitable projective space any hyperplane corresponds to a hyperplane in $\mathbb{P}^n_V$ through $P$, and one can apply the usual Bertini theorem there. Although one does not obtain information about the intersection at the point $P$, this suffices in the case $X$ has good reduction, see \ref{smoothcase}. For the general case of Theorem \ref{goodetalemap} we proceed as follows: By regularity, the intersection of a generic hyperplane through $P$ with $X$ is regular at $P$. To deal with the other points we use the Bertini theorem for abstract smoothness. Ultimately, what makes this argument work is the nice projective embedding of the blow up of $X$ at $P$ provided by \ref{closedimmersion}.

In \S5 we apply the previous results to construct good neighbourhoods of points on smooth projective $V$-schemes.

In \S6 we construct the $K(\Pi,1)$-neighbourhoods of points of smooth $V$-schemes. To finish we explain the relevance to $p$-adic nearby cycles.

\section{Sharp, tame and smooth morphisms of log schemes}

\subsection{References and notation}
As a general reference for logarithmic geometry we will use K. Kato's foundational articles \cite{log}, although for a result on differentials we refer to Ogus' notes \cite{logbook} for lack of other reference. All monoids are commutative with unit element denoted by $1$ (with the exception of $\mathbb{N}$ and $\mathbb{Z}$ whose unit element is traditionally denoted by $0$). Log structures are taken for the \'etale topology. Given a morphism of monoids $Q\to P$ and a $\mathbb{Z}[Q]$-scheme $Y$ we write $Y_Q[P]:=Y\times_{\Spec(\mathbb{Z}[Q])}\Spec(\mathbb{Z}[P])$.

\subsection{Log points}\label{setup2}
Let $(k,Q)$ be a \emph{log point}, i.e. $k$ is field with a monoid $Q$ defining the log structure via the map $Q\to k$ defined
\[ q \mapsto
\begin{cases}
1 & q=1 \\
0 & q\neq 1.
\end{cases} \]
A \emph{monogenic log point} is the case where $Q$ is a monogenic monoid, i.e. $Q$ can be generated by a single element. The \emph{standard log point} is the case where $Q=\mathbb{N}$.

We consider log points as log schemes, so that a morphism of log points just means a morphism of the associated log schemes. An \emph{extension of log points} $(k,Q)\to (L,P)$ is a morphism of the log points, together with a map of monoids $Q\to P$ giving a chart for this morphism.

\subsection{Sharp morphisms}\label{sharp}
Let $f:(X,M)\to (Y,N)$ be a morphism of log schemes. Let $x\to X$ be a geometric point, and $y=f(x)$. We say that $f$ is \emph{sharp at $x$} if there is a chart $Q\to P$ of $f$ at $x$ such that $P\to k(x)$ and $Q\to k(y)$ are log points.

We say that $f$ is \emph{sharp} if $f$ is sharp at all $x\to X$.

\begin{proposition}\label{sharplemma}
Let $f:(X,M)\to (Y,N)$ be a morphism of log schemes.
\begin{enumerate}[(i)]
\item If $f$ is sharp at a geometric point $x\to X$, then, in the above notation, $P$ and $Q$ are sharp monoids and $(k(y),Q)\to (k(x),P)$ is an extension of log points.
\item If $N$ is the trivial log structure and $M$ is fine and saturated, then $f$ is sharp.
\item If $f$ is sharp and $g:Z\to X$ is a morphism of schemes, then $f\circ g:(Z,g^{\ast}M)\to (Y,N)$ is sharp.
\item If $f$ is sharp and $g:Z\to Y$ is a morphism of schemes, then $g^*(f):(X\times_YZ,(f^*(g))^{\ast}M)\to (Z,g^*N)$ is sharp.
\item If $f$ is sharp at a geometric point $x\to X$ and $M$ is integral (resp. fine, resp. saturated), then, in the above notation, $P$ is integral (resp. fine, resp. saturated).
\end{enumerate}
\end{proposition}
\begin{proof}
(i), (iii) and (iv) follow easily from the definition. For (ii), pick a geometric point $x\to X$ and let $P=M_x/\mathcal{O}_{X,x}^{\ast}$. Then it is well-known that the map $M_x\to P$ has a section $P\to M_x$ such that the log structure associated to $P\to \mathcal{O}_X$ is the log structure $\alpha:M\to \mathcal{O}_X$ in a neighbourhood of $x$. This is obtained by choosing a section of $M^{\gp}_x\to P^{\gp}$ which exists because $P^{\gp}$ is a free abelian group (since $M$ is fine and saturated). Let $p\in P$. If $p\in\alpha^{-1}(\mathcal{O}_{X,x}^{\ast})\cong \mathcal{O}_{X,x}^{\ast}$, then  $p=1$ by construction. So $P\to k(x)$ is a log point. The fact that $P$ is sharp follows immediately from its definition. Since $N$ is the trivial log structure, we may take $Q=\{1\}$ and this proves that $f$ is sharp at $x$.

For (v), it suffices to note that the canonical map $P\oplus\mathcal{O}_{X,x}^{\ast}\to M_x$ is an isomorphism.
\end{proof}

The main case of interest in this paper is when $(Y,N)$ is a log point or the spectrum of discrete valuation ring with its canonical log structure. See \ref{exsharp} below for an example of a sharp morphism over the standard log point, and \ref{satz1} for a whole class of log schemes over a monogenic log point which are sharp.

\subsection{Example}\label{exsharp}
Suppose $X$ is a $k$-scheme with a geometric point $x\to X$ such that locally for the \'etale topology at $x$, $X$ is isomorphic to the spectrum of the ring
\[ \dfrac{k(x)[T_1,...,T_{d+1}]}{(\prod_{i=1}^rT_i^{e_i})} \]
for some $e_i\in\mathbb{N}$, with $T_i(x)=0$ for $i=1,...,d+1$. We can define a log structure in a neighbourhood $U$ of $x$ by the map
\[ \mathbb{N}^r\to\mathcal{O}_U:(n_1,..,n_r)\mapsto\prod_{i=1}^rT_i ^{n_i} \]
such that if we give $k$ the standard log point structure (\ref{setup2}), then the map
\[ \mathbb{N}\to\mathbb{N}^r:1\mapsto (e_1,...,e_r) \]
defines a morphism of log schemes $(U,\mathbb{N}^r)\to (k,\mathbb{N})$.

\begin{lemma}
The morphism $(U,\mathbb{N}^r)\to (k,\mathbb{N})$ is sharp at $x$.
\end{lemma}
\begin{proof}
It suffices to check that the induced map $\mathbb{N}^r\to k(x)$ defines the structure of a log point and that the diagram
\[ \begin{CD}
\mathbb{N} @>>> \mathbb{N}^r \\
@VVV @VVV \\
k @>>> \mathcal{O}_{X,x}
\end{CD} \]
commutes. The latter is obvious and since $T_i(x)=0$ for all $i$, $(k(x),\mathbb{N}^r)$ is a log point.
\end{proof}

\subsection{Tame morphisms}\label{tame}
Let $f:(X,M)\to (Y,N)$ be a morphism of log schemes. We say that $f$ is \emph{tame} if the sheaf of abelian groups $M^{\gp}/(f^{\ast}N)^{\gp}$ has no torsion divisible by any of the residue characteristics of $Y$. (For any sheaf of monoids $M$ we write $M^{\gp}$ for the associated sheaf of abelian groups.)\\

The interest of tame morphisms is that they provide a large class of sharp morphisms over a monogenic log point, as the following result shows.

\begin{proposition}\label{satz1}
Suppose that $f:(X,M)\to (k,Q)$ is a morphism of log schemes with $(X,M)$ fine and saturated and $(k,Q)$ a monogenic log point. For any geometric point $x\to X$ there is a fine saturated and sharp monoid $P$ defining $M$ in a neighbourhood of $x$, such that:
\begin{enumerate}[(i)]
\item $P\to k(x)$ is a log point
\item if $f$ is tame, then there is a map $Q\to P$ defining a chart of $f$ at $x$.
\end{enumerate}
In particular, if $f$ is tame, then $f$ is sharp.
\end{proposition}
\begin{proof}
We let $P=M_x/\mathcal{O}_{X,x}^*$ and then (i) can be proved as in \ref{sharplemma} (ii).

For (ii), let $q\in Q$ be a generator. Let $p_1,...,p_r$ be a set of generators of $P^{\gp}\cong\mathbb{Z}^r$ and fix a section $s:P^{\gp}\to M_x^{\gp}$. Since $M_x^{\gp}=P^{\gp}\oplus\mathcal{O}_{X,x}^{\ast}$, we may write $q=u\prod_{i=1}^rs(p_i)^{n_i}$ for some $u\in\mathcal{O}_{X,x}^{\ast}$. Let $n=(n_1,...,n_r)$ be the greatest common divisor. Since $f$ is tame, $n$ is prime to $\cha(k)$. Let $v\in\mathcal{O}_{X,x}^*$ satisfy $v^n=u$. There are integers $a_1,...,a_r$ such that $\sum_{i=1}^ra_in_i=n$. Then we have $q=\prod_{i=1}^r(v^{a_i}s(p_i))^{n_i}$. The section $s':P^{\gp}\to M_{x}^{\gp}$ defined by $s'(p_i):= v^{a_i}s(p_i)$ is such that $q$ lies in $s'(P^{\gp})\cap M_x=s'(P)$. The induced map $Q\to P$ is the desired chart.
\end{proof}

See Example \ref{cx} for an example of a non-tame log smooth scheme over the standard log point which is not sharp.

\subsection{Log smooth morphisms}\label{logsmoothdef}
Recall that a morphism $f:(X,M)\to (Y,N)$ of fine log schemes is \emph{log smooth} if and only if, locally for the \'etale topology, there exists a chart $Q\to P$ of $f$ with $P$ and $Q$ fine monoids such that
\begin{enumerate}[(a)]
\item $\ker(Q^{\gp}\to P^{\gp})$ and the torsion subgroup of $\cok(Q^{\gp}\to P^{\gp})$ are finite groups of order invertible on $Y$
\item the induced map $X\to Y_Q[P]$ is \'etale.
\end{enumerate}
$f$ is \emph{log \'etale} if, in addition, $\cok(Q^{\gp}\to P^{\gp})$ is a finite group of order invertible on $Y$.

\begin{remark} In spite of the above characterization, log smooth morphisms are not necessarily tame (see \ref{cx} for an example).
\end{remark}

\subsection{Log differentials}
Let $f:(X,M)\to (Y,N)$ be a morphism of fine log schemes. We write $\omega^1_{(X,M)/(Y/N)}$ for the module of relative logarithmic differentials.

\subsubsection{}\label{diff}
Given any morphism $g:(Y,N)\to (Z,L)$ of fine log schemes there is a right-exact sequence of sheaves of relative logarithmic differentials
\[ \begin{CD}
0 @>>> f^{\ast}\omega^1_{(Y,N)/(Z,L)} @>>> \omega^1_{(X,M)/(Z,L)} @>>> \omega^1_{(X,M)/(Y,N)} @>>> 0
\end{CD} \]
and $f$ is log smooth if $g\circ f$ is log smooth and the sequence is left-exact and locally split (\cite[3.12]{log}).

\subsubsection{}\label{jacobi}
Let $i:Z\hookrightarrow X$ be a closed immersion of ideal sheaf $\mathcal{I}$ and give $Z$ the inverse image log structure of $(X,M)$ (making $i$ a strict closed immersion). Then there is a right-exact sequence
\[ \begin{CD}
0 @>>>\mathcal{I}/\mathcal{I}^2 @>>> i^{\ast}\omega^1_{(X,M)/(Y,N)} @>>> \omega^1_{(Z,i^{\ast}M)/(Y,N)} @>>> 0
\end{CD} \]
and if $f$ is log smooth, then $f\circ i$ is log smooth if and only if the sequence is left-exact and locally split. See \cite[IV, 3.2.2]{logbook} for the proof.

\subsection{Generic log smoothness}\label{generic}
Let $f:(X,M)\to (Y,N)$ be a morphism of finite presentation of fine log schemes with $Y$ irreducible. Assume moreover there is a chart $Q\to N$ with $Q$ a fine monoid. Let $\xi\in Y$ be the generic point, endowed with the inverse image log structure of $(Y,N)$. We claim that if $f^{-1}(\xi)$ is log smooth over $\xi$, then there exists a dense open $V\subset Y$ such that $f|_V$ is log smooth. We may assume $Y$ is quasi-compact. For each geometric point $x\to f^{-1}(\xi)$ there is an \'etale neighbourhood $Z^x\to f^{-1}(\xi)$ of $x$, a chart $P^x\to M|_{Z^x}$ with $P^x$ a fine monoid, and a map $Q\to P^x$ such that $Z^x\to \xi_Q[P^x]$ is \'etale and the map $Q\to P^x$ satisfies property (a) of \ref{logsmoothdef}. Moreover, since $M$ is fine, $P^x$ defines the log structure of $X$ in an \'etale neighbourhood $X^x$ of $x$. Since $f$ is of finite presentation, up to replacing $X^x$ by an \'etale neighbourhood of $x$, there is an open neighbourhood $Y^x$ of $\xi$ in $Y$ and a morphism of finite presentation $X^x\to (Y^x)_Q[P^x]$. Namely, we may take $Y^x$ to be a dense open subset of the image of $X^x\to X\overset{f}{\to}Y$, up to replacing $X^x$ by $X^x\times_YY^x$. Since this morphism becomes \'etale at the point $\xi\in Y^x$, by \cite[IV, 17.7.8]{ega} there is a dense open $U^x\subset Y^x$ such that the map $X^x\times_YU^x\to (U^x)_Q[P^x]$ is \'etale. Now since $f$ is of finite presentation, there is a finite number of points $x_1,...,x_n\in f^{-1}(\xi)$ such that the map $\coprod_{i=1}^nX^{x_i}\times_Y\xi\to f^{-1}(\xi)$ is an \'etale covering. If $U:=U^{x_1}\cap U^{x_2}\cap\cdots\cap U^{x_n}$, then there is a dense open $V\subset U$ such that $\coprod_{i=1}^nX^{x_i}\times_YV\to X\times_YV$ is an \'etale covering, and this proves the claim.

\subsection{Notation}
In the sequel we will usually drop the log structures from notation when it is clear which log structures are meant. Also we will always write $\omega^1_{X/Y}$ for logarithmic differentials, as opposed to the usual differentials $\Omega^1_{X/Y}$.

\section{Hyperplane sections}

Let $k$ be a field.

\subsection{Affine hyperplane sections}\label{setup}
Let $X$ be a $k$-scheme of finite type and $f:X\to\mathbb{A}^n_k$ a morphism of $k$-schemes defined by global sections $f_1,...,f_n\in\Gamma(X,\mathcal{O}_X)$. Let
\[ Z:=\underline{\Spec}_X\left(\mathcal{O}_X[U_0,...,U_n]/(U_0+U_1f_1+U_2f_2+...+U_nf_n)\right). \]
Then we have a projection morphism $q:Z\to\mathbb{A}^{n+1}_k$ whose fibre over a point $u=(u_0,...,u_n)$ is the fibre $f^{-1}(H_u)$, where $H_u$ is the hyperplane of equation $u_0+u_1T_1+u_2T_2+...+u_nT_n=0$ where $T_1,...,T_n$ are the global coordinates on $\mathbb{A}^n_k$ with images $f_1,...,f_n$ respectively in $\Gamma(X,\mathcal{O}_X)$.

There is a natural isomorphism
\[ Z\cong X\times_k\mathbb{A}^n_k \]
obtained by the map
\begin{eqnarray*}
\mathcal{O}_X[U_0,...,U_n]/(U_0+U_1f_1+...+U_nf_n) &\to &\mathcal{O}_X[U_1,...,U_n]\\
U_0 &\mapsto& -\sum_{i=1}^nU_if_i
\end{eqnarray*}
so that $Z$ can viewed as trivial vector bundle on $X$.

Let $\mathbb{K}$ be the function field of $\mathbb{A}^{n+1}_k$. We write $Z_\mathbb{K}$ for the generic fibre of $q:Z\to\mathbb{A}^{n+1}_k$, $X_\mathbb{K}:=X\otimes_k\mathbb{K}$, and $i_\mathbb{K}:Z_\mathbb{K}\hookrightarrow X_\mathbb{K}$ for the closed immersion deduced from $i$. Note that by \cite[I, 6.3.18]{bertini} $q$ is a dominant morphism if and only if $\dim\overline{f(X)}>0$, hence $Z_{\mathbb{K}}$ is non-empty if and only if $\dim\overline{f(X)}>0$.

We refer to \cite[I, 6]{bertini} for further details.

\subsection{Bertini theorem for log smoothness}
Let $(k,Q)$ be a log point, with $Q$ a fine monoid. We endow $\mathbb{A}^{n+1}_k$ with the inverse image log structure of $(k,Q)$ via the canonical morphism $\mathbb{A}^{n+1}_k\to\Spec(k)$. If $(X,M)\to (k,Q)$ is a morphism of log schemes, then we have the fibre product $X\times_{k}\mathbb{A}^{n+1}_k$ in the category of log schemes (the log structure is just the inverse image of $M$ under the projection $X\times_{k}\mathbb{A}^{n+1}_k\to X$). Finally, we give $i:Z\hookrightarrow X\times_{k}\mathbb{A}^{n+1}_k$ the inverse image log structure.

Endow $\mathbb{K}$ with the inverse image log structure of $\mathbb{A}^{n+1}_k$, $X_{\mathbb{K}}$ the inverse image log structure of $X\times_k\mathbb{A}^{n+1}_k$, and $i_\mathbb{K}:Z_\mathbb{K}\hookrightarrow X_\mathbb{K}$ with the inverse image log structure of $X_{\mathbb{K}}$.

\begin{theorem}\label{main}
With the above notation and hypothesis.  Assume that $(X,M)$ is a sharp log $(k,Q)$-scheme. If $f$ is unramified, then for any $z\in Z_\mathbb{K}$ the sequence
\[ \begin{CD}
0 @>>> k(z) @>{1\mapsto \sum_{i=1}^nU_idf_i}>> i_\mathbb{K}^{\ast}\omega^1_{X_\mathbb{K}/\mathbb{K}}\otimes_{\mathcal{O}_{Z_{\mathbb{K}}}}k(z) @>>> \omega^1_{Z_\mathbb{K}/\mathbb{K}}\otimes_{\mathcal{O}_{Z_{\mathbb{K}}}}k(z) @>>> 0
\end{CD} \]
is exact.
\end{theorem}

This theorem will be proved below. From it one derives the Bertini theorem for log smoothness.

\begin{corollary}\label{logbertinigeneric}
With notation and hypothesis as in \ref{main}. If $X$ is log smooth over $(k,Q)$, then $Z_\mathbb{K}$ is log smooth over $(\mathbb{K},Q)$.
\end{corollary}
\begin{proof}
Choose $z\in Z_{\mathbb{K}}$. By \ref{main} there is an ideal $I\subsetneq\mathcal{O}_{Z_{\mathbb{K}},z}$ such that the sequence
\[ \begin{CD}
0 @>>> \mathcal{O}_{Z_{\mathbb{K}},z}/I @>{1\mapsto \sum_{i=1}^nU_idf_i}>> (i_\mathbb{K}^{\ast}\omega^1_{X_\mathbb{K}/\mathbb{K}})_z @>>> (\omega^1_{Z_{\mathbb{K}}/\mathbb{K}})_z @>>> 0
\end{CD} \]
is exact. Since $\omega^1_{X_\mathbb{K}/\mathbb{K}}\cong\omega^1_{X/k}\otimes_k\mathbb{K}$ is locally free, \ref{main} implies that $\text{Tor}_1^{\mathcal{O}_{Z_{\mathbb{K}},z}}((\omega^1_{Z_{\mathbb{K}}/\mathbb{K}})_z,k(z))=0$, i.e. $(\omega^1_{Z_{\mathbb{K}}/\mathbb{K}})_z$ is flat, hence so is $\mathcal{O}_{Z_{\mathbb{K}},z}/I$. Thus, tensoring the exact sequence
\[ 0\to I\to \mathcal{O}_{Z_{\mathbb{K}},z}\to \mathcal{O}_{Z_{\mathbb{K}},z}/I\to 0 \]
with $k(z)$ we get an exact sequence
\[ 0\to I\otimes_{\mathcal{O}_{Z_{\mathbb{K}},z}}k(z)\to k(z)\to k(z)\to 0 \]
and therefore $I\otimes_{\mathcal{O}_{Z_{\mathbb{K}},z}}k(z)=0$, whence $I=0$ by Nakayama's lemma. It now follows from \ref{jacobi} that $Z_{\mathbb{K}}$ is log smooth over $(\mathbb{K},Q)$ at $z$.
\end{proof}

\begin{corollary}\label{logbertini}
With notation and hypothesis as in \ref{main}. If $X$ is quasi-compact and $(X,M)\to (k,Q)$ is log smooth, then there is a dense open $V\subset \mathbb{A}^{n+1}_k$ such that $q^{-1}(V)\to V$ is log smooth over $(k,Q)$.
\end{corollary}
\begin{proof}
This follows from \ref{generic}.
\end{proof}

\begin{proof}[Proof of \ref{main}]
The proof is in the style of \cite[I, 6.3]{bertini}. We may assume that $k$ is algebraically closed. We first show a lemma.

\begin{lemma}\label{satz2}
Let $(k,Q)\to (L,P)$ be an extension of log points. Then the natural map
\[ \Omega^1_{L/k}\to\omega^1_{L/k} \]
has a section, in particular is injective.
\end{lemma}
\begin{proof}
Let $N=Q\oplus k^{\ast}$ and $M=P\oplus L^{\ast}$ be the log structures. Note that if $(p,u)\in P\oplus L^{\ast}$ then the image of $(p,u)$ in $L$ is 0 unless $p=1$. Let $d:L\to\Omega^1_{L/k}$ be the exterior derivative and let $\partial:M\to\Omega^1_{L/k}$ be the homomorphism of monoids $M=P\oplus L^{\ast}\to\Omega^1_{L/k}:(p,u)\mapsto u^{-1}du$. Since the map $N\to M$ maps $Q$ to $P$, one checks easily that the pair $(d,\partial)$ forms a log derivation of $(L,M)$ to $\Omega^1_{L/k}$ over $(k,N)$ (\cite[5.1]{logsmooth}). So, by \cite[5.3]{logsmooth}, there is a unique $L$-linear map $s:\omega^1_{L/k}\to\Omega^1_{L/k}$ such that $d=s\circ d$ and $\partial=s\circ\dlog$, where $\dlog:M\to\omega^1_{L/k}$ is the canonical map (the hypothesis that the log structures be fine is superfluous, cf. \cite[IV, 1.1.6]{logbook}). In particular, the composition of $s$ with the canonical map $\Omega^1_{L/k}\to\omega^1_{L/k}$ maps $dx\in\Omega^1_{L/k}$ to itself, hence is the identity map by linearity.
\end{proof}

Let $z\in Z_{\mathbb{K}}$ and suppose for a contradiction that the map
\[ \begin{CD}
k(z) @>{1\mapsto \sum_{i=1}^n=U_idf_i}>> \omega^1_{X_\mathbb{K}/\mathbb{K}}\otimes_{\mathcal{O}_{X_\mathbb{K}}}k(z)
\end{CD} \]
is zero. Let $x\in X$ be the image of $z\in Z_\mathbb{K}$ and let $d=\td(k(x)/k)$ be the transcendence degree of $k(x)$ over $k$. Note that $d>0$: otherwise $Z_{\mathbb{K}}\times_X\Spec(k(x))=\Spec(k(U_0,...,U_n)/(\sum_iU_if_i))=\emptyset$ and there is no point $z$ lying above $x$.

Choose an algebraic closure $\overline{k(x)}$ of $k(x)$ and let $k(x)^{\text{sep}}$ be the separable algebraic closure of $k(x)$ in $\overline{k(x)}$. Let $\bar{x}\to X$ be the geometric point lying above $x$ corresponding to the inclusion $k(x)\subset k(x)^{\text{sep}}$. Since $X$ is sharp over $k$, by definition (\ref{sharp}) there is a chart $Q\to P$ of $f$ at $\bar{x}$ such that the induced map $P\to k(\bar{x})$ is a log point. Then the map $(k,Q)\to (k(\bar{x}),P)$ is an extension of log points, so by \ref{satz2} the quotient $\omega^1_{k(\bar{x})/k}\cong\omega^1_{k(x)/k}\otimes_{k(x)}k(\bar{x})$ of $\omega^1_{X/k}\otimes_{\mathcal{O}_X}k(\bar{x})$ contains $\Omega^1_{k(\bar{x})/k}\cong\Omega^1_{k(x)/k}\otimes_{k(x)}k(\bar{x})$. Thus, the canonical map $\Omega^1_{k(x)/k}\to\omega^1_{k(x)/k}$ is injective. Hence $\sum_iU_idf_i=0$ in $\Omega^1_{k(x)/k}\otimes_{k(x)}k(z)$. Since $k$ is algebraically closed, it follows that $d=\dim_{k(x)}\Omega^1_{k(x)/k}$. Since $f:X\to\mathbb{A}^n_k$ is unramified, $\Omega^1_{k(x)/k}$ is generated by $df_1,...,df_n$, so without loss of generality we may assume that $df_1,...,df_d$ form a $k(x)$-basis of $\Omega^1_{k(x)/k}$. Thus, for $i=d+1,...,n$ we may write $df_i=\sum_{j=1}^da_{ij}df_j$ for some $a_{ij}\in k(x)$. Then the relation $\sum_iU_idf_i=0$ gives the following equations
\[ U_j+\sum_{i=d+1}^nU_ia_{ij}=0 \]
for each $j=1,...,d$. So we deduce that $U_1,...,U_d\in k(z)$ belong to the $k(x)$-vector space generated by $U_{d+1},U_{d+2},...,U_n$. Moreover, since $U_0=-\sum_{i=1}^nU_if_i$, it follows that
\[ k(z)=k(x)(U_0,U_1,...,U_n)=k(x)(U_{d+1},U_{d+2},...,U_n) \]
and so
\[ \td(k(z)/k(x))\leq n-d. \]
On the other hand we have
\begin{eqnarray*}
\td(k(z)/k(x)) &=& \td(k(z)/k)-\td(k(x)/k) \\
&\geq & n+1-d
\end{eqnarray*}
which is the desired contradiction.
\end{proof}

\begin{cx}\label{cx}
We give an example of a non-sharp (and non-tame) log smooth scheme over the standard log point $(k,\mathbb{N})$ whose generic hyperplane section is nowhere log smooth. Let $0\neq p=\cha(k)$ and consider the morphism of monoids
\begin{eqnarray*}
h:\mathbb{N} &\to & \mathbb{N}\oplus\mathbb{Z} \\
1 &\mapsto & (p,1).
\end{eqnarray*}
Then $\cok(h^{\gp})\cong\mathbb{Z}$, hence
\[ X:=\Spec(k)_{\mathbb{N}}[\mathbb{N}\oplus\mathbb{Z}]\cong\Spec\left(k[t,u^{\pm 1}]/(t^p)\right) \]
with its natural log structure is log smooth over $(k,\mathbb{N})$. Let $M$ denote the log structure on $X$, $N$ the standard log point structure on $k$, and $f:(X,M)\to (k,N)$ the morphism given by map $h$. Note that the image of $(1,0)\in\mathbb{N}\oplus\mathbb{Z}$ in $M^{\gp}/(f^{\ast}N)^{\gp}$ is $p$-torsion: we have $p(1,0)=(p,0)=h(1)+(0,-1)$ and $(0,-1)\in\mathcal{O}_X^*$. So $f$ is not tame.

We claim that $f$ is not sharp. If not, then let $x\to X$ be a geometric point and a chart $Q:=\mathbb{N}\to P$ of $f$ at $x$ such that $P\to k(x)$ is a log point. Then we have an isomorphism $P\oplus k(x)^*\cong M_{k(x)}$. On the other hand, one sees easily that $M_{k(x)}=\mathbb{N}\oplus k(x)^*$, where $\mathbb{N}\to k(x)$ is the standard log point with $\mathbb{N}$ generated by $t$. From this it follows that $P\cong\mathbb{N}$ as monoids, hence if $n\in P$ is a generator, its image in $M_{k(x)}$ is of the form $tv$ for some $v\in k(x)^*$. Now let $q=1\in\mathbb{N}=Q$ be the generator. By definition, its image in $M_{k(x)}$ is $t^pu$. On the other hand, its image in $P$ is $(tv)^m$ for some $m$. Thus, we must have $m=p$ and $u=v^p$. In particular, if the image of $x$ in $X$ is the generic point, then $k(x)$ is a separable closure of $k(u)$ so this is impossible. This proves the claim.

Now set $f_1=t,f_2=u,f_3=u^{-1}$ and use the same notation as before. Pick a point $z\in Z_{\mathbb{K}}$ with image $x\in X$. Then in $\omega^1_{X/k}$ we have
\[ 0=\dlog(q)=\dlog(p,1)=p\dlog(1,0)+\dlog(0,1)=\dlog(0,1)=\dlog(u) \]
so $du=0$ in $\omega^1_{X/k}$. Since $dt=t\dlog(1,0)$ and $t(x)=0$, it follows that
\[ U_1dt+U_2du+U_3d(u^{-1})=0 \]
in $\omega^1_{X_\mathbb{K}/\mathbb{K}}\otimes_{\mathcal{O}_{X_\mathbb{K}}}k(z)$, so a generic hyperplane section of $X$ is nowhere log smooth over $(k,\mathbb{N})$.
\end{cx}

\subsection{A generalization of the Bertini theorem for smoothness}\label{bertinigeneral}
Let again $f:X\to\mathbb{A}^n_k$ and $Z$ be as in \ref{setup}. Here we do not assume the presence of any log structures. Assume there is a morphism of finite type $m:\mathcal{X}\to X$ with $\dim \mathcal{X}\leq d$ and a surjective map of $\mathcal{O}_\mathcal{X}$-modules
\[ m^*\Omega^1_{X/k}\twoheadrightarrow\mathcal{E} \]
with $\mathcal{E}$ a vector bundle of rank $d$ on $\mathcal{X}$. Define $\mathcal{Z}=\mathcal{X}\times_XZ$. As before we let $\mathbb{K}=k(U_0,...,U_n)$. Now, the map
\[ \begin{CD} \mathcal{O}_Z @>{1\mapsto\sum_{i=1}^nU_idf_i}>> \Omega^1_{X/k}\otimes_{\mathcal{O}_X}\mathcal{O}_Z \end{CD} \]
extends to a map
\[ \mathcal{O}_\mathcal{Z}\overset{\phi}{\to}\mathcal{E}\otimes_{\mathcal{O}_\mathcal{X}}\mathcal{O}_\mathcal{Z}. \]
Set
\[ \mathcal{F}:=\cok(\phi) \]
the cokernel of the map $\phi$. So we have a commutative diagram with exact rows
\[ \xymatrix{
\mathcal{O}_{\mathcal{Z}} \ar[r] \ar@{=}[d] & \Omega^1_{X/k}\otimes_{\mathcal{O}_{X}}\mathcal{O}_{\mathcal{Z}} \ar[d] \ar[r] & \Omega^1_{Z/k}\otimes_{\mathcal{O}_Z}\mathcal{O}_{\mathcal{Z}} \ar[d] \ar[r] & 0 \\
\mathcal{O}_{\mathcal{Z}} \ar[r]^{\phi} & \mathcal{E}\otimes_{\mathcal{O}_\mathcal{X}}\mathcal{O}_\mathcal{Z} \ar[r] & \mathcal{F} \ar[r] & 0
}
\]

\begin{theorem}\label{bertgeneral}
If $f$ is unramified, then $\mathcal{F}_{\mathbb{K}}$ is a rank $d-1$ vector bundle quotient of $\Omega^1_{Z_{\mathbb{K}}/\mathbb{K}}\otimes_{\mathcal{O}_{Z_{\mathbb{K}}}}\mathcal{O}_{\mathcal{Z}_{\mathbb{K}}}$. Moreover, if $\mathcal{Z}_{\mathbb{K}}\neq\emptyset$, then $\dim\mathcal{Z}_{\mathbb{K}}\leq d-1$.
\end{theorem}
\begin{proof}
Note that the last statement follows from \cite[I, 6.3]{bertini}. For the first statement we will show that $\mathcal{F}_{\mathbb{K}}$ is a vector bundle of rank $d-1$ by the same argument as for the classical Bertini theorem for smoothness (loc. cit.). Since $\mathcal{E}$ is flat, it suffices to show that the map
\[ \phi(z):k(z)\to\mathcal{E}\otimes_{\mathcal{O}_\mathcal{X}}k(z) \]
is injective at any point $z\in \mathcal{Z}_{\mathbb{K}}$. Let $x\in \mathcal{X}$ be the image of $z$. Since $f$ is unramified, $df_1,...,df_n$ generate $\Omega^1_{X/k}$, and since $\mathcal{E}$ is a quotient of the latter on $\mathcal{X}$, we may assume that the images of $df_1,...,df_d$ form a $k(x)$-basis of $\mathcal{E}\otimes_{\mathcal{O}_\mathcal{X}}k(x)$. So for $i=d+1,...,n$ we can write $df_i=\sum_{j=1}^da_{ij}df_j$ for some $a_{ij}\in k(x)$. If $\phi(z)$ is not injective, i.e. $\sum_{i=1}^nU_idf_i=0$ in $\mathcal{E}\otimes_{\mathcal{O}_\mathcal{X}}k(z)$, then we get equations
\[ U_j+\sum_{i=d+1}^nU_ja_{ij}=0 \]
for $j=1,...,d$. Together with $U_0+\sum_{i=1}^nU_if_i=0$, these equations imply that
\[ k(z)=k(x)(U_{d+1},...,U_n) \]
hence
\[ \td(k(z)/k(x))\leq n-d. \]
On the other hand, since $\dim \mathcal{X}\leq d$ we have $\td(k(x)/k)\leq d$, hence
\[ \td(k(z)/k(x))=\td(k(z)/k)-\td(k(x)/k)\geq n+1-d \]
a contradiction.
\end{proof}

The theorem easily implies the following generalization of Bertini's theorem for smoothness (the classical case being $\mathcal{X}=X$ and $\mathcal{E}=\Omega^1_{X/k}$).

\begin{corollary}
If $f$ is unramified and $X$ quasi-compact, then there is a dense open $U\subset\mathbb{A}^{n+1}_k$ such that $\mathcal{F}_U$ is a rank $d-1$ vector bundle quotient of $\Omega^1_{Z_U/U}\otimes_{\mathcal{O}_{Z_U}}\mathcal{Z}_U$.
\end{corollary}
\begin{proof}
It suffices to find $U$ such that $\mathcal{F}_U$ is a vector bundle of rank $d-1$. First assume that $\mathcal{F}_{\mathbb{K}}$ is free of basis $e_1,...,e_{d-1}$ say. Then we can find a dense open set $U'\subset\mathbb{A}^{n+1}_k$ such that $e_i\in\mathcal{F}_{U'}$ for $i=1,...,d-1$, and this defines a map $\mathcal{O}_{\mathcal{Z}_{U'}}^{d-1}\to\mathcal{F}_{U'}$. Since the cokernel $\mathcal{G}$ of this map vanishes over the generic point of $\mathbb{A}^{n+1}_k$, by a limit argument we can find a dense open $U\subset U'$ such that $\mathcal{G}_U=0$. Since $X$, hence $\mathcal{Z}$, is quasi-compact, $\mathcal{Z}\to\mathbb{A}^{n+1}_k$ is of finite presentation, so the general case follows easily from this.
\end{proof}

\subsection{Multiple hyperplane sections}\label{multiple}
We can generalize the previous results to the case of $r$ hyperplanes. This time we let
\[ Z:=\underline{\Spec}_{X}\left(\dfrac{\mathcal{O}_{X}[U_i^{(j)}]_{0\leq i\leq n, 1\leq j\leq r}}{\left(U_0^{(j)}+\sum_{i=1}^nU_i^{(j)}f_i\right)_{j=1,...,r}}\right) \]
and $q:Z\to\mathbb{A}^{(n+1)r}_k$ the projection. The generic fibre of $q$ will be non-empty if (and only if) $\dim\overline{f(X)}\geq r$; namely if the latter condition holds then there is an irreducible component $Y$ of $X$ for which $\dim\overline{f(Y)}\geq r$ and then $Z\times_XY\cong Y\times_k\mathbb{A}^{nr}_k$ is an irreducible component of $Z\cong X\times_k\mathbb{A}^{nr}_k$ and the generic fibre of $q|_{Z_Y}$ is non-empty by \cite[I, 6.6]{bertini}.

\begin{corollary}\label{bertini1}
Endow $\mathbb{A}^{(n+1)r}_k$ with the inverse image log structure from $\Spec(k)$ and $Z$ with the inverse image log structure of $X$ (via the structure maps). If $f$ is unramified and $X\to\Spec(k)$ is sharp log smooth, then the generic fibre of $q$ is log smooth.
\end{corollary}
\begin{proof}
The proof is by induction on $r$, the case $r=1$ being of course \ref{logbertinigeneric}. Let $\mathbb{K}'=k(U_i^{(j)})_{0\leq i\leq n, 1\leq j\leq r-1}$ and consider
\[ Z':=\underline{\Spec}_{X}\left(\dfrac{\mathcal{O}_{X}\otimes_k\mathbb{K}'}{\left(U_0^{(j)}+\sum_{i=1}^nU_i^{(j)}f_i\right)_{j=1,...,r-1}}\right). \]
By induction, $Z'$ is log smooth over $(\mathbb{K}',Q)$. Moreover, it is sharp by \ref{sharplemma} (iii) and (iv). So by \ref{logbertinigeneric} the scheme
\[ Z''=\underline{\Spec}_{Z'}\left(\dfrac{\mathcal{O}_{Z'}\otimes_{\mathbb{K}'}\mathbb{K}'(U_0^{(r)},U_1^{(r)},...,U_n^{(r)})}{\left(U_0^{(r)}+\sum_{i=1}^nU_i^{(r)}f_i\right)}\right) \]
is log smooth over the field $\mathbb{K}'(U_0^{(r)},U_1^{(r)},...,U_n^{(r)})=k(U_i^{(j)})_{0\leq i\leq n, 1\leq j\leq r}$. To complete the proof it suffices to remark that this is just the generic fibre of $q$.
\end{proof}

Let $\mathcal{X}$ be as in \ref{bertinigeneral}. Define $\mathcal{Z}=\mathcal{X}\times_XZ$ with $Z$ as above. There is a canonical map $\phi:\mathcal{N}\otimes_{\mathcal{O}_Z}\mathcal{O}_{\mathcal{Z}}\to\mathcal{E}\otimes_{\mathcal{O}_{\mathcal{X}}}\mathcal{O}_{\mathcal{Z}}$, where $\mathcal{N}$ is the conormal sheaf of $Z$ in $X\times\mathbb{A}^{(n+1)r}$. Set $\mathcal{F}:=\cok(\phi)$.

\begin{corollary}\label{bertini2}
With notation and assumptions as in \ref{bertinigeneral} except $\mathcal{Z}$ and $\mathbb{K}:=k(U_i^{(j)})_{0\leq i\leq n, 1\leq j\leq r}$. If $f$ is unramified, then $\mathcal{F}_{\mathbb{K}}$ is a rank $d-r$ vector bundle quotient of $\Omega^1_{Z_{\mathbb{K}}/\mathbb{K}}\otimes_{\mathcal{O}_{Z_{\mathbb{K}}}}\mathcal{O}_{\mathcal{Z}_{\mathbb{K}}}$ and $\dim\mathcal{Z}_{\mathbb{K}}\leq d-r$ if $\mathcal{Z}_{\mathbb{K}}\neq\emptyset$.
\end{corollary}
\begin{proof}
The proof is an easy induction on $r$ as in that of \ref{bertini1} and we use the same notation. The case $r=1$ being \ref{bertgeneral}, assume $r>1$. Let $\mathcal{Z}'=\mathcal{X}\times_XZ'$ and $\mathcal{Z}''=\mathcal{X}\times_XZ''$. Note that $\mathcal{Z}''=\mathcal{Z}_{\mathbb{K}}$. We may assume $\mathcal{Z}_{\mathbb{K}}\neq\emptyset$. By induction on $r$ we have a rank $d-(r-1)$ vector bundle quotient of $\Omega^1_{Z'/\mathbb{K}'}\otimes_{\mathcal{O}_{Z'}}\mathcal{O}_{\mathcal{Z}'}$ and since $\mathcal{Z}_{\mathbb{K}}\neq\emptyset$ we must also have $\mathcal{Z}'\neq\emptyset$, hence $\dim\mathcal{Z}'\leq d-(r-1)$ by induction hypothesis. Applying \ref{bertgeneral}, we get a rank $d-r$ vector bundle quotient of $\Omega^1_{Z''/\mathbb{K}}\otimes_{\mathcal{O}_{Z''}}\mathcal{O}_{\mathcal{Z}''}$. Since it is the quotient of $\mathcal{E}\otimes_{\mathcal{O}_{\mathcal{X}}}\mathcal{O}_{\mathcal{Z}''}$ by the submodule generated by the images of the differentials $\sum_iU_i^{(j)}df_i$ for $j=1,...,r$, it is equal to $\mathcal{F}_{\mathbb{K}}$. Finally, since $\mathcal{Z}''\neq\emptyset$ we have $\dim\mathcal{Z}''\leq d-(r-1)-1=d-r$ by \cite[I, 6.3]{bertini}, and this completes the induction.
\end{proof}

\section{Quasi-projective case}
Let $T$ be an arbitrary irreducible affine noetherian scheme and let $f:X\to\mathbb{P}^n_T$ be a morphism of finite type.

\subsection{Grassmannians}\label{setupgrass}
Let $G_{\mathbb{P}^n_T,r}:=\Grass_{n+1-r}(\Gamma(\mathbb{P}^n_T,\mathcal{O}_{\mathbb{P}^n_T}(1)))$ be the grassmannian of quotients of $\Gamma(\mathbb{P}^n_T,\mathcal{O}_{\mathbb{P}^n_T}(1))\simeq\mathcal{O}_T^{n+1}$ which are locally free of rank $n+1-r$. Let $S$ be a $T$-scheme and $l_1,...,l_r\in\mathcal{O}_S(S)^{n+1}$. We say that the $l_i$ are \emph{non-degenerate} if the element $\wedge_{i=1}^rl_i\in\wedge_{i=1}^r\mathcal{O}_S^{n+1}$ generates a direct summand. In this case it is clear that the $l_i$ induce a locally free quotient $Q$ of $\mathcal{O}_S^{n+1}$ of rank $n+1-r$, i.e. an element in $G_{\mathbb{P}^n_T,r}(S)$. Conversely, if $S$ is a $T$-scheme over which every vector bundle is free, then to give a quotient of $\mathcal{O}_S^{n+1}$ or rank $n+1-r$ is the same as giving a free submodule $M\subset\mathcal{O}_S^{n+1}$ of rank $r$ which has a non-degenerate basis $l_1,...,l_r$. In any case, if $l_1,...,l_r$ are non-degenerate then each of the $l_i$ determines a hyperplane in $\mathbb{P}^n_S$ so that $L=\mathbb{P}_S(Q)\hookrightarrow\mathbb{P}^n_S$ is the subscheme cut out by the $l_i$, $1\leq i\leq r$.

Fix a projective space $\mathbb{P}^n_T$ and $r\in\mathbb{N}$ and write $G:=G_{\mathbb{P}^n_T,r}$. Let $\mathcal{Q}$ be the tautological quotient on $G$. There is a canonical closed immersion
\[ \Grass_1(\mathcal{Q})\hookrightarrow\Grass_1(\mathcal{O}^{n+1}_{G})\cong\Grass_1(\mathcal{O}^{n+1}_T)\times_TG=\mathbb{P}^n_T\times_TG \]
whose image is the set of points $(x,L)\in\mathbb{P}^n_T\times_TG$ such that $x\in L$.
So if $f:X\to\mathbb{P}^n_T$ is a morphism, then the functor
\[ \mathcal{H}_X=\{ (x,L)\in X\times_TG: f(x)\in L \}  \]
is the scheme $\mathcal{H}_X:=\Grass_1(\mathcal{Q})\times_{\mathbb{P}^n_T}X$. It comes equipped with a projection
\[ p:\mathcal{H}_X\to G. \]
In particular, if $k$ is a field and $P\in\mathbb{P}^n(k)$, then the linear subvariety defined by a point $l\in \mathcal{H}_P(k)$ can be identified with the intersection of a set of $r$ non-degenerate $k$-hyperplanes through the point $P$, and the map $p:\mathcal{H}_P\to G$ is a closed immersion.

\begin{lemma}\label{smoothincidencevariety}
Fix a fine log structure on $T$. If $X$ is log smooth over $T$, then the scheme $\mathcal{H}_X$ (endowed with the inverse image log structure of $X$) is log smooth over $T$. In particular, $\mathcal{H}_P$ is smooth over $k$.
\end{lemma}
\begin{proof}
We must check that for a nilpotent exact closed immersion $S_0\hookrightarrow S$ of log $T$-schemes any element in $\mathcal{H}_X(S_0)$ lifts to an element of $\mathcal{H}_X(S)$, up to replacing $S$ by an \'etale covering. Now we have
\[ \mathcal{H}_X(S_0)=\{(x,L)\in X(S_0)\times G(S_0):f(x)\in L(S_0)\} \]
and let us translate the meaning of a point $(x,L)\in \mathcal{H}_X(S_0)$ in terms of modules. The point $L\in G(S_0)$ corresponds to a quotient
\[ \mathcal{O}_{S_0}^{n+1}\twoheadrightarrow Q_0 \]
with $Q_0$ locally free of rank $n+1-r$, the point $f(x)\in\mathbb{P}^n(S_0)$ corresponds to a quotient
\[ \mathcal{O}_{S_0}^{n+1}\twoheadrightarrow P_0 \]
with $P_0$ locally free of rank one, and finally the condition $(x,L)\in \mathcal{H}_X(S_0)$ means that the quotient onto $P_0$ factors
\[ \mathcal{O}_{S_0}^{n+1}\twoheadrightarrow Q_0\to P_0. \]
Now, $X$ being log smooth, up to localizing for the \'etale topology on $S_0$ we can find a point $x'\in X(S)$ lifting $x$. Then $f(x')$ corresponds to a quotient
\[ \mathcal{O}_{S}^{n+1}\twoheadrightarrow P \]
with $P$ locally free of rank one. Let $M:=\ker\left(\mathcal{O}_{S}^{n+1}\twoheadrightarrow P\right)$, $N_0=\ker\left(\mathcal{O}_{S_0}^{n+1}\twoheadrightarrow Q_0\right)$. Up to localizing on $S_0$ we may assume that $N_0$ is a free module on the basis $l_1,...,l_r$. Since $N_0\subset M\otimes_{\mathcal{O}_{S}}\mathcal{O}_{S_0}$, we can choose lifts $n_1,...,n_r\in M$ of $l_1,...,l_r$. Then $n_1,...,n_r\in\mathcal{O}_{S}^{n+1}$ form part of a basis (by Nakayama's lemma), so if we let $N$ be the free submodule generated by the $n_i$, then $Q:=\mathcal{O}_{S}^{n+1}/N$ is a free module. We take $Q$ to be the lift of $Q_0$. By construction we have that the map $\mathcal{O}_{S}^{n+1}\twoheadrightarrow P$ factors
\[ \mathcal{O}_{S}^{n+1}\twoheadrightarrow Q\to P. \]
This means that $f(x')$ lies on the linear space $L'$ corresponding to $Q$, so the point $(x',L')\in \mathcal{H}_X(S)$ is a lift of $(x,L)$.
\end{proof}

\subsection{Reduction to the affine case}
Let $M_1,...,M_{n+1\choose r}$ be the $r\times r$ minors of the matrix $\left(U_i^{(j)}\right)_{0\leq i\leq n,1\leq j\leq r}$. Consider the open $Y\subset\mathbb{A}^{(n+1)r}_T$ defined
\[ Y:=\mathbb{A}^{(n+1)r}_T-\bigcap_{i=1}^{{n+1\choose r}}V(M_i) \]
where $V(M_i)\subset\mathbb{A}^{(n+1)r}_T=\Spec(\mathcal{O}_T[U_i^{(j)}])$ is the zero set of $M_i$. Then for any $T$-scheme $S$ the set $Y(S)$ is a subset of the set of affine homogeneous $S$-hyperplanes $l_1,...,l_r$ which are non-degenerate: after making choice $X_0,...,X_n$ of homogeneous coordinates an $S$-point
\[ (u_0^{(1)},...,u_n^{(1)},...,u_0^{(r)},...,u_n^{(r)})\in Y(S) \]
corresponds to the hyperplanes $\sum_{i=0}^nu_i^{(j)}X_i=0$ for $1\leq j\leq r$. This identification depends on the choice of homogeneous coordinates. Fixing such a choice, let $Q$ be the quotient of $\oplus_{i=0}^n\mathcal{O}_YX_i$ by the $\mathcal{O}_Y$-submodule generated by the elements $l_j:=\sum_{i=0}^nU_i^{(j)}X_i$ for $1\leq j\leq r$. Since, by definition, at every point of $Y$ some $r\times r$ minor of the matrix $(U_i^{(j)})$ is invertible it follows that $l_1\wedge l_2\wedge\cdots\wedge l_r$ forms part of a basis of $\wedge^r(\oplus_{i=0}^n\mathcal{O}_YX_i)$, which implies that $Q$ is a vector bundle of rank $n+1-r$ on $Y$. This defines a morphism
\[ g:Y\to G:=G_{\mathbb{P}^n_T,r}. \]

\begin{lemma}\phantomsection\label{dominant}
\begin{enumerate}[(i)]
\item The morphism $g:Y\to G$ is surjective.
\item $G_t$ is irreducible for all $t\in T$.
\end{enumerate}
\end{lemma}
\begin{proof}
Let $S$ be the spectrum of a field and $l\in G(S)$. Since $S$ is the spectrum of a field, there is a set $\{l_1,...,l_r\}$ of non-degenerate homogeneous $S$-hyperplanes such that $l$ is the quotient of $\oplus_{i=0}^n\mathcal{O}_SX_i$ by the submodule generated by $l_1,...,l_r$. For each $1\leq j\leq r$, write $l_j=\sum_{i=0}^nu_i^{(j)}X_i$. These hyperplanes are non-degenerate if and only if some $r\times r$-minor of the matrix $(u_i^{(j)})$ is non-zero. Hence the point $(u_i^{(j)})\in\mathbb{A}^{(n+1)r}(S)$ lies in $Y(S)$, and this proves (i). Since $Y_t$ is irreducible, (i) implies (ii).
\end{proof}

\begin{proposition}\label{subspaceclosed}
Assume $T$ is the spectrum of a field $k$. Let $\mathcal{Q}$ denote the tautological quotient of $\Gamma(\mathbb{P}^n_k,\mathcal{O}(1))\otimes_k\mathcal{O}_G$. If $S\subset\Gamma(\mathbb{P}^n_k,\mathcal{O}(1))$ is a subspace of dimension at most $n+1-r$ and $\mathcal{Q}':=\im(S\otimes_k\mathcal{O}_G\to\mathcal{Q})$, then there is a dense open $U\subset G$ such that the map $S\otimes_k\mathcal{O}_U\to\mathcal{Q}'|_U$ is an isomorphism and $\mathcal{Q}/\mathcal{Q}'|_U$ is a vector bundle on $U$.
\end{proposition}
\begin{proof}
Let $X_0,....,X_m$ be a basis of $S$ and let $X_{m+1},...,X_n$ be elements of $\Gamma(\mathbb{P}^n_k,\mathcal{O}(1))$ such that $X_0,...,X_n$ form a basis. Let $g:Y\to G$ be the morphism determined by these homogeneous coordinates. It suffices to check that the map $S\otimes_k\mathcal{O}_G\to\mathcal{Q}'$ is an isomorphism over the generic point of $G$. By \ref{dominant} (i), for this we may base change to the fraction field $k(Y)$ of $Y$.

Let $R=\Gamma(\mathbb{P}^n_k,\mathcal{O}(1))/S$. Then we have a commutative diagram with exact rows
\[ \xymatrix{
0 \ar[r] & S\otimes_kk(Y) \ar[r] \ar[d] & \oplus_{i=0}^n k(Y)X_i \ar[r] \ar[d] & R\otimes_kk(Y) \ar[r] \ar[d] & 0 \\
0 \ar[r] & \mathcal{Q}'\otimes_{\mathcal{O}_G}k(Y) \ar[r] & \mathcal{Q}\otimes_{\mathcal{O}_G}k(Y) \ar[r] & (\mathcal{Q}/\mathcal{Q}')\otimes_{\mathcal{O}_G}k(Y) \ar[r] & 0 \\
} \]
and, by definition of the morphism $g$, the kernel $\kappa$ of the middle vertical map is generated by the elements $\sum_{i=0}^{n}U_i^{(j)}X_i$ for $j=1,...,r$. If $0\neq s\in\kappa\cap (S\otimes k(Y))$, then we may write
\[ s=\sum_{j=1}^ra_j\sum_{i=0}^{n}U_i^{(j)}X_i \]
and
\[ s=\sum_{i=0}^mb_iX_i \]
for some $a_j,b_i\in k(Y)$. Equating we find
\[ \sum_{j=1}^ra_jU_i^{(j)}=0 \]
for $i=m+1,...,n$. This implies that the rank of the matrix $U:=(U_i^{(j)})_{m+1\leq i\leq n,1\leq j\leq r}$ is strictly less than $r$. Since $r\leq n-m$ by assumption, all $r\times r$-minors of $U$ vanish. But the subset of $Y$ where all $r\times r$-minors of $U$ vanish is a proper closed subset, so this cannot happen. Therefore, $\kappa\cap (S\otimes k(Y))=0$.
\end{proof}

We now fix a choice of homogeneous coordinates and let $\mathbb{A}^n_T\subset\mathbb{P}^n_T$ be a standard open with induced coordinates $T_1,...,T_n$. Let $X'=X\times_{\mathbb{P}^n_T}\mathbb{A}^n_T$ and define
\[ Z:=\underline{\Spec}_{X'}\left(\dfrac{\mathcal{O}_{X'}[U_i^{(j)}]_{0\leq i\leq n, 1\leq j\leq r}}{\left(U_0^{(j)}+\sum_{i=1}^nU_i^{(j)}f_i\right)_{j=1,...,r}}\right)
\]
where $f_i\in\Gamma(X',\mathcal{O}_{X'})$ is the image of $T_i$ for $i=1,...,n$. Let $i:Z\hookrightarrow X'\times_T\mathbb{A}^{(n+1)r}_T$ be the closed immersion defining $Z$ and $q:Z\to\mathbb{A}^{(n+1)r}_T$ be the second projection. ($Z$ is the scheme associated to $X'$ in \ref{multiple}, except that we assumed there that $T$ was the spectrum of a field.)

\begin{proposition}\label{reduction}
With our fixed choice of coordinates.
\begin{enumerate}[(i)]
\item There is a cartesian square
\[\xymatrix{
Z\times_{\mathbb{A}^{(n+1)r}_T}Y \ar[d]^{q} \ar[r] & \mathcal{H}_{X'} \ar[d]^{p} \\
Y \ar[r]^{g} & G
} \]
\item The closed immersion $i\times_{\mathbb{A}^{(n+1)r}_T}Y:Z\times_{\mathbb{A}^{(n+1)r}_T}Y\hookrightarrow X'\times_TY$ deduced from $i$ is equal, via the identification $Z\times_{\mathbb{A}^{(n+1)r}_T}Y=\mathcal{H}_{X'}\times_GY$ of (i), to the closed immersion $\mathcal{H}_{X'}\times_GY\hookrightarrow X'\times_TY$ deduced from the closed immersion $\mathcal{H}_{X'}\hookrightarrow X'\times_kG$ defining $\mathcal{H}_{X'}$.
\end{enumerate}
\end{proposition}
\begin{proof}
By definition, $Z_Y:=Z\times_{\mathbb{A}^{(n+1)r}_T}Y$ is the closed subscheme of $X'\times_TY$ defined by the ideal generated by $U_0^{(j)}+\sum_{i=1}^{n}U_i^{(j)}f_i$ for $j=1,...,r$. So for any $T$-scheme $S$, every element of $Z_Y(S)$ determines points $x\in X'(S)$ and $(u_0^{(1)},...,u_n^{(1)},...,u_0^{(r)},...,u_n^{(r)})\in Y(S)$ such that the image of $x$ in $\mathbb{A}^n_T$ lies on the (non-degenerate) hyperplanes given by $u_0^{(j)}+\sum_{i=1}^{n}u_i^{(j)}T_i=0$ for $1\leq j\leq r$, and conversely. This is obviously the same thing as $\left(\mathcal{H}_{X'}\times_{G}Y\right)(S)$. Since this correspondence is clearly functorial in $S$, this proves (i).

For (ii), an element of $P\in Z_Y(S)$ determines points $x\in X'(S)$ and $(u_0^{(1)},...,u_n^{(1)},...,u_0^{(r)},...,u_n^{(r)})\in Y(S)$, and then $i(P)=(x,(u_0^{(1)},...,u_n^{(1)},...,u_0^{(r)},...,u_n^{(r)}))$. This is clearly equal to the image of $P$ under the composition $Z_Y=\mathcal{H}_{X'}\times_GY\to X'\times_TY$.
\end{proof}

This allows us to show the following Bertini theorem for log smoothness by reducing to the affine case proven before. In the following we assume $T=\Spec(k)$ with $k$ a field.

\begin{corollary}\label{bertini6}
Assume $T$ has the structure of a log point. Let $f:X\to\mathbb{P}^n_k$ be unramified with $X$ a sharp log smooth $k$-scheme. Endow $\mathcal{H}_X$ with the inverse image log structure of $X$ via the second projection $\pr_2:\Grass_1(\mathcal{Q})\times_{\mathbb{P}^n_T}X\to X$, and $G$ with the inverse image log structure from the structure morphism to $T$. There is a dense open $U\subset G$ such that $p|_U:\mathcal{H}_X|_{U}\to U$ is log smooth.
\end{corollary}
\begin{proof}
For each standard open $\mathbb{A}^n_k\subset\mathbb{P}^n_k$, let $X'=X\times_{\mathbb{P}^n_k}\mathbb{A}^n_k$, so that we have a commutative diagram as in \ref{reduction}. Endow $\mathbb{A}^{(n+1)r}_T$ with the inverse image log structure of $T$ and $Z$ with the inverse image log structure from $X$ so that we are in the situation of \ref{bertini1}. It is clear that with these log structures the diagram of \ref{reduction} (i) is a cartesian square of log schemes. Hence by \ref{bertini1} and \ref{dominant}, the generic fibre of $\mathcal{H}_{X'}\to G$ becomes log smooth after the field extension $k(G)\to k(Y)$. Since the log structures on $k(Y)$ and $k(G)$ are both induced by $Q\to k$, it now follows easily from \ref{jacobi} that the generic fibre of of $\mathcal{H}_{X'}\to G$ is itself log smooth. So by \ref{generic} there is a dense open subset $U'\subset G$ such that the restriction of $\mathcal{H}_{X'}\to G$ to $U'$ is log smooth. Hence we can take $U$ to be intersection of the (finitely many) $U'$ for each standard open $\mathbb{A}^n_k\subset\mathbb{P}^n_k$.
\end{proof}

We also have the quasi-projective version of \ref{bertini2}.

\begin{corollary}\label{bertini7}
Let $f:X\to\mathbb{P}^n_k$ be unramified and let $m:\mathcal{X}\to X$ be a morphism of finite type with $\dim\mathcal{X}\leq d$ such that there is a surjective map of $\mathcal{O}_{\mathcal{X}}$-modules
\[ \pi:m^*\Omega^1_{X/k}\twoheadrightarrow\mathcal{E} \]
where $\mathcal{E}$ is vector bundle of rank $d$ on $\mathcal{X}$. If $I$ denotes the ideal sheaf of $\mathcal{H}_X$ in $X\times_kG$, and
\[ \phi:I/I^2\otimes_{\mathcal{O}_X}\mathcal{O}_{\mathcal{X}}\to\mathcal{E}\otimes_{\mathcal{O}_{X}}\mathcal{O}_{\mathcal{H}_X} \]
the composition of the natural map $\left(I/I^2\to\Omega^1_{X/k}\otimes_{\mathcal{O}_X}\mathcal{O}_{\mathcal{H}_X}\right)\otimes_{\mathcal{O}_X}\mathcal{O}_{\mathcal{X}}$ with the map $\pi\otimes_{\mathcal{O}_X}\mathcal{O}_{\mathcal{H}_X}$, then there is a dense open $U\subset G$ such that $\cok(\phi)|_U$ is a vector bundle of rank $d-r$ on $\mathcal{H}_{\mathcal{X}}\times_GU$.
\end{corollary}
\begin{proof}
Since $G$ is irreducible it suffices to show that the restriction of $\cok(\phi)$ to $\mathcal{H}_{\mathcal{X}}\times_G\Spec(k(G))$ is a vector bundle of rank $d-r$. We may check this locally and therefore replace $X$ by $X'$. By \ref{reduction} we may reduce to the situation of \ref{multiple} and apply \ref{bertini2}.
\end{proof}

\section{Hyperplanes through a point over a discrete valuation ring}
Let $V$ be a complete discrete valuation ring with uniformizer $\pi$, algebraically closed residue field $k=V/\pi V$, fraction field $K=V[1/\pi]$, and set $T:=\Spec(V)$. We will often make the following abuse of notation: since a point $L\in G_{\mathbb{P}^n_T,r}(K)=G_{\mathbb{P}^n_T,r}(T)$ can be identified with a linear subvariety of $\mathbb{P}^n_T$, for a subscheme $X\subset \mathbb{P}^n_T$ we write $X\cap L$ for the scheme-theoretic intersection of $X$ with the linear subvariety defined by $L$.

\subsection{Specialization of points}\label{special}
Let $X$ be a separated $T$-scheme of finite type and $x\in X_K$ be a point. We define the \emph{specialization} $\spe_X(x)$ of $x$ to be $\overline{\left\{x\right\}}\cap X_k$, where $\overline{\left\{x\right\}}$ is closure of $x$ in $X$. We simply write $\spe(x)$ for $\spe_X(x)$ when no ambiguity can occur.

We gather here some facts on specialization.

\begin{lemma}\label{closed}
If $x\in X_K$ is a closed point, then either $\spe(x)=\emptyset$ or $\spe(x)$ is a closed point. Conversely, if $\spe(x)$ is a closed point, then $x$ is closed.
\end{lemma}
\begin{proof}
Since $X$ is a separated $T$-scheme of finite type, by a theorem of Nagata there is a proper $T$-scheme $\bar{X}$ and an open immersion $X\hookrightarrow\bar{X}$. Let $x\in X_K$ be closed. Then there is a finite extension $K\subset L$ such that $x\in X(L)\subset\bar{X}(L)$. Then $x$ gives a point in $s\in\bar{X}(V_L)$, where $V_L$ is the normalization of $V$ in $L$. Since both $\Spec(V_L)$ and $\bar{X}$ are proper over $T$, $s(\Spec(V_L))$ is closed in $\bar{X}$. It follows that the closure of $x$ in $\bar{X}$ is equal to $s(\Spec(V_L))$. So $\spe(x)$ is isomorphic to a subscheme of $\Spec(V_L\otimes_Vk)$, i.e. is a closed point or is empty.

For the converse, let $U=\Spec(A)\subset X$ be an affine open neighbourhood of $\spe(x)\in X$. If $x\notin U$, then $x\in X-U$, so $\spe(x)\in\overline{\left\{x\right\}}\subset X-U$, which is absurd. Hence $x\in U$, and $\overline{\left\{x\right\}}\cap U$ corresponds to a quotient $A\to B$ with $B$ reduced. Then $\Spec(B)$ is irreducible (otherwise $x$ lies on an irreducible component which is a proper closed subset of $\overline{\left\{x\right\}}$), so $\Spec(B)$ is integral. Moreover, the reduction of $B\otimes_{V}k$ is $k$. Since $B\otimes_VK\neq 0$, by \cite[IV, 14.3.10]{ega} we have $\dim(B\otimes_{V}K)=0$. By Noether normalization $B\otimes_{V}K$ is a finite $K$-algebra, so since $B\otimes_VK$ is a domain it is a finite field extension of $K$.
\end{proof}

\begin{lemma}\label{gen}
Let $x\in X_k$ be a closed point. The set $\spe^{-1}(x)$ is equal to the set of closed points of $\Spec(\mathcal{O}_{X,x}\otimes_VK)$, viewed as a subset of $X_K$.
\end{lemma}
\begin{proof}
Let $U$ be an open neighbourhood of $x\in X_k$. If $y\in X_K$ is a closed point such that $\spe(y)=x$ and $y\notin U$, then $x\in\overline{\left\{y\right\}}\not\subset U$, a contradiction. Hence $U$ contains $\spe^{-1}(x)$. Since $\Spec(\mathcal{O}_{X,x}\otimes_VK)=\cap_{U\ni x}U_K$, where the intersection is taken over all open neighbourhoods of $x$, it follows that $\Spec(\mathcal{O}_{X,x}\otimes_VK)$ contains $\spe^{-1}(x)$.

For the converse, pick a closed point of $y\in\Spec(\mathcal{O}_{X,x}\otimes_VK)$. It gives a quotient $\mathcal{O}_{X,x}\otimes_VK\to L$ where $L=K(y)$ is the residue field at $y$. We  first claim that $L$ is a finite extension of $K$. Let $C$ be the image of $\mathcal{O}_{X,x}$ in $L$. It is a local ring with residue field $k(x)=k$ satisfying $C[1/\pi]=L$. So by a theorem of Artin-Tate (\cite[0, 16.3.3]{ega}) $C$ is a local domain of dimension at most 1. Since $V$ is universally catenary, by \cite[IV, 5.6.4]{ega} we have $\dim V+\td(L/K)=\dim C+\td(k(x)/k)\leq 1$, hence $\td(L/K)=0$ as claimed.

Now for every small enough affine open $\Spec(A)=U\subset X$ containing $x$, $y$ induces a quotient $A\otimes_VK\to L$. We claim that the image $V'$ of $A$ in $L$ has special fibre over $V$ consisting of a single point. Note that $\pi$ is not a unit in $V'$ (otherwise $\pi$ would be a unit in $C$). Now, since $V'$ is an integral domain and $V$ is henselian, it suffices to show that $V'$ is a finite $V$-module, for then we know that $\Spec(V'\otimes_Vk)$ is connected and finite, hence a single point. To see this, first note that since $V'$ is a one-dimensional domain ($\dim(V'_{\mathfrak{p}})\leq\dim(V)+\td(L/K)=1$ for any prime ideal $\mathfrak{p}\subset V'$), $V'\otimes_{V}k$ is a zero-dimensional $k$-algebra of finite type, so by Noether normalization it is finite. Lifting a $k$-basis $e_1,...,e_n$ of $V'\otimes_{V}k$ we see that $V'\subset \sum_{i=1}^nVe_i+\pi V'$, hence $V'\subset\sum_{i=1}^nVe_i+\pi^mV'$ for all $m\geq 1$. Since $V'$ is $\pi$-adically separated (by Krull's Intersection Theorem) and $V$ complete, this implies that $V'\subset\widehat{V'}=\sum_{i=1}^nVe_i$, whence $V'=\sum_{i=1}^nVe_i$. This proves the claim. Note that the specialization $\spe(y)$ of $y$ in $U$ is given by the quotient $A\to V'\otimes_Vk$, hence $\spe(y)$ is a closed point in $U_k$. Since this holds for every small enough neighbourhood $U$ of $x$, it follows that $\spe(y)=x$, as required.
\end{proof}

\begin{corollary}\label{sp}
Let $x\in X_k$ be closed. An open subset $U_0\subset X_K$ is of the form $U_0=U_{K}$ for some open neighbourhood $U$ of $x$ if and only if $\spe^{-1}_X(x)\subset U_0$.
\end{corollary}
\begin{proof}
We first show the if part. Let $Y_0$ be the complement of $U_0$ in $X_{K}$ and let $Y$ be its closure in $X$. If $x\in Y$, then by the last lemma any closed point of $y\in\Spec(\mathcal{O}_{Y,x}\otimes_{V}K)$ is a closed point of $Y_0$ satisfying $\spe_X(y)=x$. Since $\spe^{-1}_X(x)\subset U_0$ by assumption, this is absurd. So $U:=X-Y$ contains $x$. For the converse, it suffices to note that if $x\in U$, then by the last lemma $\spe^{-1}_X(x)\subset U$.
\end{proof}

\subsection{Hyperplanes through a point}\label{hyperplanespoint}
Let $P\in\mathbb{P}^n(k)$ and let $I_P\subset\mathcal{O}_{\mathbb{P}^n_T}$ be its ideal sheaf. For a sheaf $\mathcal{F}$ on $\mathbb{P}^n_T$ we write $\Gamma(\mathcal{F}):=\Gamma(\mathbb{P}^n_T,\mathcal{F})$ to simplify. Then for any integer $d$ we have an exact sequence
\[ 0\to I_P(d)\to\mathcal{O}_{\mathbb{P}^n_T}(d)\to k(P)\to 0 \]
so taking global sections we get a left-exact sequence
\begin{equation}\label{exseq:globalideal}
0\to \Gamma(I_P(d))\to\Gamma(\mathcal{O}_{\mathbb{P}^n_T}(d))\to k(P)\to 0.
\end{equation}

\begin{lemma}\label{hyperplanelocus}
If $d\geq 0$, then the sequence \ref{exseq:globalideal} is exact.
\end{lemma}
\begin{proof}
It suffices to show $\Gamma(I_P(d))\neq\Gamma(\mathcal{O}_{\mathbb{P}^n_T}(d))$. For $d=0$ this is obvious, so assume $d\geq 1$. If $\Gamma(I_P(d))=\Gamma(\mathcal{O}_{\mathbb{P}^n_T}(d))$ is bijective, then every hypersurface of degree $d$ vanishes at $P$, equivalently every hyperplane in $\mathbb{P}(\Gamma(\mathcal{O}(d)))$ vanishes on the image of $P$. But in a suitable choice of coordinates we have $P=(1:0:\cdots:0)\in\mathbb{P}(\Gamma(\mathcal{O}(d)))\otimes_Vk$ and the hyperplane given by the first coordinate does not vanish at $P$. So $\Gamma(I_P(d))\neq\Gamma(\mathcal{O}_{\mathbb{P}^n_T}(d))$.
\end{proof}

Since $H^1(\mathbb{P}^n_T,\mathcal{O}_{\mathbb{P}^n_T}(d))=0$ for $d\geq 0$ we immediately find

\begin{corollary}
$H^1(\mathbb{P}^n_T,I_P(d))=0$ for $d\geq 0$.
\end{corollary}

Since $I_P(d)$ is $\pi$-torsion free, taking cohomology in the exact sequence
\[ 0\to I_P(d)\overset{\cdot\pi}{\to} I_P(d)\to I_P(d)\otimes_Vk\to 0 \]
we obtain

\begin{corollary}\label{corhyperplanelocus}
$\Gamma(I_P(d))\otimes_Vk=\Gamma(I_P(d)\otimes_Vk)$ for $d\geq 0$.
\end{corollary}

For $d\geq 0$ we can tensor the exact sequence \ref{exseq:globalideal} with $k$ to obtain exact sequences
\begin{equation}\label{exseq:globalideal1} 0\to k(P)\to\Gamma(I_P(d))\otimes_Vk\to\Gamma(I_P(d)\cdot\mathcal{O}_{\mathbb{P}^n_k})\to 0
\end{equation}
\begin{equation}\label{exseq:globalideal2} 0\to \Gamma(I_P(d)\cdot\mathcal{O}_{\mathbb{P}^n_k})\to \Gamma(\mathcal{O}_{\mathbb{P}^n_k}(d))\to k(P)\to 0.
\end{equation}

\begin{proposition}\label{global}
$I_P(d)$ and $I_P(d)\cdot\mathcal{O}_{\mathbb{P}^n_k}$ are generated by global sections for $d\geq 1$.
\end{proposition}
\begin{proof}
We first show that $I_P(d)\cdot\mathcal{O}_{\mathbb{P}^n_k}$ is generated by global sections for $d\geq 1$. By \cite[Lecture 14, Prop.]{mumford} it suffices to check that $I_P\cdot\mathcal{O}_{\mathbb{P}^n_k}$ is $d$-regular, i.e. $H^i(\mathbb{P}^n_k,I_P\cdot\mathcal{O}_{\mathbb{P}^n_k}(d-i))=0$ for $i\geq 1$. To see this, note that since $d\geq 1$, by \ref{exseq:globalideal2} for $i\geq 1$ we have $H^i(\mathbb{P}^n_k,I_P\cdot\mathcal{O}_{\mathbb{P}^n_k}(d-i))=H^i(\mathbb{P}^n_k,\mathcal{O}_{\mathbb{P}^n_k}(d-i))$, and by Serre's computation of the cohomology of projective space we know that this is zero for $0<i<n$, and that $H^n(\mathbb{P}^n_k,\mathcal{O}_{\mathbb{P}^n_k}(d-n))$ is the $k$-dual of $H^0(\mathbb{P}^n_k,\mathcal{O}_{\mathbb{P}^n_k}(-d-1))=0$.

For $I_P(d)$, since $\mathbb{P}^n_T$ is proper over $T$, by \ref{corhyperplanelocus} and Nakayama's lemma it will suffice to show that $I_P(d)\otimes_Vk$ is generated by global sections if $d\geq 1$. Now, we have an exact sequence of $\mathcal{O}_{\mathbb{P}^n_k}$-modules
\[ 0\to k(P)\to I_P(d)\otimes_Vk\to I_P(d)\cdot\mathcal{O}_{\mathbb{P}^n_k}\to 0 \]
and since $k(P)$ is a skyscraper sheaf and $I_P(d)\cdot\mathcal{O}_{\mathbb{P}^n_k}$ is globally generated for $d\geq 1$, by \ref{exseq:globalideal1} this implies the result.
\end{proof}

\begin{corollary}\label{diagram}
There is a commutative diagram of canonical morphisms
\[ \xymatrix{
\mathbb{P}^n_T\setminus\{P\} \ar[r] & \mathbb{P}(\Gamma(I_P(2))) \\
\mathbb{P}^n_k\setminus\{P\} \ar[r] \ar[u] & \mathbb{P}(\Gamma(I_P(2)\cdot\mathcal{O}_{\mathbb{P}^n_k})) \ar[u]
} \]
where the vertical maps are closed immersions.
\end{corollary}
\begin{proof}
Let us define the morphisms. For the top map, by \ref{global} on $\mathbb{P}^n_T$ we have the quotient
\[ \Gamma(I_P(2))\otimes_V\mathcal{O}_{\mathbb{P}^n_T}\twoheadrightarrow I_P(2) \]
and $I_P(2)$ is isomorphic to the line bundle $\mathcal{O}_{\mathbb{P}^n_T}(2)$ on $\mathbb{P}^n_T\setminus\{P\}$, so we get a corresponding morphism $\phi:\mathbb{P}^n_T\setminus\{P\}\to \mathbb{P}(\Gamma(I_P(2)))$. Similarly for the bottom map. It is clear (cf. \ref{exseq:globalideal}, \ref{exseq:globalideal1}, \ref{exseq:globalideal2}) that these maps fit in a commutative diagram as above.
\end{proof}

In fact, the horizontal maps of the diagram of \ref{diagram} are immersions. This is a consequence of the following key result.

\begin{proposition}\label{key}
Let $\tilde{\mathbb{P}}$ be the blow up of $\mathbb{P}^n_T$ at the point $P$, and $E\subset\tilde{\mathbb{P}}$ the exceptional divisor.
\begin{enumerate}[(i)]
\item There is a canonical closed immersion $\phi:\tilde{\mathbb{P}}\to\mathbb{P}(\Gamma(I_P(2)))$ extending the morphism of \ref{diagram}.
\item $\phi(\tilde{\mathbb{P}})\cap\mathbb{P}(\Gamma(I_P(2)\cdot\mathcal{O}_{\mathbb{P}^n_k}))$ is isomorphic to the blow up of $\mathbb{P}_k^n$ at $P$.
\item $\phi(E)$ is a linear subvariety of $\mathbb{P}(\Gamma(I_P(2)))\otimes_Vk$.
\end{enumerate}
\end{proposition}
\begin{proof}
We first define the morphism $\phi$. Since $\Gamma(I_P(2))$ generates $I_P(2)$ and $I_P\cdot\mathcal{O}_{\tilde{\mathbb{P}}}$ is invertible, $\mathcal{L}:=I_P(2)\cdot\mathcal{O}_{\tilde{\mathbb{P}}}$ is a rank 1 locally free quotient of $\Gamma(I_P(2))\otimes_V\mathcal{O}_{\tilde{\mathbb{P}}}$, i.e. a morphism $\tilde{\mathbb{P}}\to\mathbb{P}(\Gamma(I_P(2)))$. It clearly extends the map of \ref{diagram}.

Let $x_0,...,x_n\in\Gamma(\mathcal{O}_{\mathbb{P}^n_T}(1))$ be homogeneous coordinates such that $P$ is the point $(1:0:\cdots:0)\in\mathbb{P}^n(k)$. Then the $x_ix_j$ ($0\leq i,j\leq n$) form a basis of $\Gamma(\mathcal{O}_{\mathbb{P}^n_T}(2))$, so one checks easily that $x_ix_j$ for $i\neq j$ together with $x_1^2,...,x_{n}^2$ and $\pi x_0^2$ form a basis of $\Gamma(I_P(2))$ (cf. \ref{exseq:globalideal}).

Now we investigate the map $\phi$ in a neighbourhood of $P$. Let $W\subset\mathbb{P}^n_T$ be the complement of vanishing set of $x_0\in\Gamma(\mathcal{O}_{\mathbb{P}^n_T}(1))$. Then $t_i:=\frac{x_i}{x_0}$ ($1\leq i\leq n$) are coordinates on $W$, and over $W$ the ideal $I_P$ is generated by the $t_i$ together with $\pi$. So $\tilde{\mathbb{P}}\times_{\mathbb{P}^n_T}W$ is covered by the spectra of the rings
\[ A_i:=\dfrac{V[t_1,...,t_n,u_{1i},...,u_{ni},s_i]}{(t_iu_{1i}-t_1,...,t_iu_{ni}-t_n,s_it_i-\pi)}=\dfrac{V[t_i,u_{ji},s_i]_{1\leq j\leq n, j\neq i}}{(s_it_i-\pi)} \]
for $1\leq i\leq n$, together with
\[ A_0:=\dfrac{V[t_1,...,t_n,s_1^{-1},...,s_n^{-1}]}{(\pi s_1^{-1}-t_1,...,\pi s_n^{-1}-t_n)}=V[s_j^{-1}]_{1\leq j\leq n} \]
(here $s_i^{-1}$ is an indeterminate, not invertible). On the other hand, consider the principal open $U_i$ of $\mathbb{P}(\Gamma(I_P(2)))$ where $x_ix_0\neq 0$ ($i\neq 0$). Let $X_i\subset\tilde{\mathbb{P}}$ be the maximal open subset where $x_ix_0$ generates $\mathcal{L}$, so that $\phi$ restricts to a morphism $X_i\to U_i$. We claim that $\Spec(A_i)\subset X_i$. Indeed, we have the trivialization
\[ \mathcal{O}_{\mathbb{P}^n_T}|_{W}\cong\mathcal{O}_{\mathbb{P}^n_T}(2)|_{W}:f\mapsto x_0^2f \]
and the ideal $I_P\cdot\mathcal{O}_{\Spec(A_i)}$ is generated by $t_i$, so $\mathcal{L}|_{\Spec(A_i)}=I_P(2)\cdot\mathcal{O}_{\Spec(A_i)}$ is generated by $t_ix_0^2$. This is precisely the image of $x_ix_0$, and this implies the claim. Now $U_i$ has coordinates $\frac{x_jx_k}{x_ix_0}$ and $\frac{\pi x_0^2}{x_ix_0}$. We have natural identifications $\phi^{\#}(\frac{x_ix_j}{x_ix_0})=t_j$, $\phi^{\#}(\frac{x_jx_0}{x_ix_0})=u_{ji}$, $\phi^{\#}(\frac{x_jx_k}{x_ix_0})=u_{ji}t_k$ (one checks easily that $u_{ji}t_k=u_{ki}t_j$ in $A_i$) and $\phi^{\#}(\frac{\pi x_0^2}{x_ix_0})=s_i$, so the map $\phi|_{U_i}:X_i\to U_i$ induces a surjective homomorphism
\[ \phi^{\#}:\Gamma(U_i,\mathcal{O}_{U_i})\twoheadrightarrow A_i. \]
Similarly, if $U_0$ is the principal open of $\mathbb{P}(\Gamma(I_P(2)))$ where $\pi x_0^2\neq 0$, and $X_0\subset\tilde{\mathbb{P}}$ the maximal open where $\pi x_0^2$ generates $\mathcal{L}$, then (using that $I_P\cdot\mathcal{O}_{\Spec(A_0)}$ is generated by $\pi$) one checks that $\Spec(A_0)\subset X_0\to U_0$. We have $\phi^{\#}(\frac{x_jx_0}{\pi x_0^2})=s_j^{-1}$ for $j\neq 0$ and $\phi^{\#}(\frac{x_jx_k}{\pi x_0^2})=t_js_k^{-1}$ for $j,k\neq 0$, so we get a surjective map
\[ \phi^{\#}:\Gamma(U_0,\mathcal{O}_{U_0})\twoheadrightarrow A_0. \]
Since it is easy to see that the principal open where $x_i^2\neq 0$ ($i\neq 0$) contains the principal open of $\tilde{\mathbb{P}}$ where $x_i\neq 0$ as a closed subscheme, this at least shows that $\phi$ is a local immersion. But $\tilde{\mathbb{P}}$ is integral and $\phi$ is proper, so by \cite[I, 8.2.8]{ega} $\phi$ is a closed immersion. This proves (i). For (ii) it suffices to note that $\mathbb{P}(\Gamma(I_P(2)\cdot\mathcal{O}_{\mathbb{P}^n_k}))$ is the zero set of $\pi x_0^2$ in $\mathbb{P}(\Gamma(I_P(2)))\otimes_Vk$ (cf. \ref{exseq:globalideal1}) and that the rings $A_i\otimes_Vk/(s_i)$ for $i\neq 0$ give an open covering of the blow up of $\mathbb{P}^n_k$ at $P$ over $W$. Finally, note that the ideal of the exceptional divisor in $A_0$ is generated by $\pi$, and the kernel of the map
\[ \Gamma(U_0,\mathcal{O}_{U_0})\otimes_Vk=k\left[\tfrac{x_jx_k}{\pi x_0^2}\right]_{j,k}\twoheadrightarrow k[s_j^{-1}]_{j}=A_0\otimes_Vk \]
is generated by the coordinates $\frac{x_jx_k}{\pi x_0^2}$ for $j,k\neq 0$. This implies that the exceptional divisor shares a dense open subset with the linear variety cut out by these coordinates, so since the exceptional divisor is integral it must be equal to this linear variety, i.e. (iii).
\end{proof}

\begin{corollary}\label{closedimmersion}
Let $X\hookrightarrow\mathbb{P}^n_T$ be an immersion and let $\tilde{X}$ be the blow up of $X$ at a closed point $P\in X_k$. There is a canonical immersion $\phi:\tilde{X}\to\mathbb{P}(\Gamma(I_P(2)))$. Moreover, if $X$ is regular at $P$, then the image of the exceptional divisor is a linear subvariety of $\mathbb{P}(\Gamma(I_P(2)))\otimes_Vk$.
\end{corollary}
\begin{proof}
Everything except the final statement is an immediate consequence of the proposition. Now if $X$ is regular at $P$ the ideal of $P$ in $X$ can be generated by a regular subsequence of a regular sequence generating the ideal of $P$ in $\mathbb{P}^n_T$. So the exceptional divisor of $\tilde{X}$ is a linear subvariety of the exceptional divisor of $\tilde{\mathbb{P}}$, and hence is linear by the proposition.
\end{proof}

This description of $\tilde{X}$ gives us good control over hyperplane sections of its special fibre in $\mathbb{P}(\Gamma(I_P(2)))$. The next lemma allows one to compare hyperplanes in $\mathbb{P}(\Gamma(\mathcal{O}_{\mathbb{P}^n_T}(2)))$ with hyperplanes in $\mathbb{P}(\Gamma(I_P(2)))$, especially for those hyperplanes which pass through the point $P$.

Recall (\S\ref{setupgrass}) that for a projective space $\mathbb{P}^n_S$ over a scheme $S$ we write $G_{\mathbb{P}^n_S,r}=\Grass_{n+1-r}(\Gamma(\mathbb{P}^n_S,\mathcal{O}_{\mathbb{P}^n_S}(1)))$.

\begin{lemma}\label{zariskiopen}
Let $G:=G_{\mathbb{P}(\Gamma(\mathcal{O}_{\mathbb{P}^n_T}(2))),r}$ and $G':=G_{\mathbb{P}(\Gamma(I_P(2))),r}$.
\begin{enumerate}[(i)]
\item There is a canonical isomorphism $\mathcal{H}_P=G_{\mathbb{P}(\Gamma(I_P(2)\cdot\mathcal{O}_{\mathbb{P}^n_k})),r}$.
\item There is a canonical isomorphism $\iota:G_K\cong G'_K$.
\item There is a dense open $G'_0\subset G'_k$ and a canonical morphism $\rho:G'_0\to \mathcal{H}_{P}$. This morphism admits a section, in particular it is surjective.
\item For any closed $L\in G'_K$ such that $\spe_{G'}(L)\in G'_0$ we have $\spe_{G}(\iota^{-1}(L))=\rho(\spe_{G'}(L))$.
\item $\iota\left(\spe^{-1}_G(\mathcal{H}_P)\cap G(K)\right)=\spe_{G'}^{-1}(G'_0)\cap G'(K)$.
\end{enumerate}
So we have a commutative diagram
\[ \xymatrix{
G'_K \ar[d]^{\spe_{G'}} & \spe^{-1}_{G'}(G'_0) \ar[d]^{\spe_{G'}} \ar@{_(->}[l] \ar[r]^{\iota^{-1}} & \spe^{-1}_G(\mathcal{H}_P) \ar@{^(->}[r] \ar[d]^{\spe_G} & G_K \ar[d]^{\spe_G} \\
G'_k & G'_0\ar@{_(->}[l]  \ar[r]^{\rho} & \mathcal{H}_P\ar@{^(->}[r] & G_k
} \]
\end{lemma}
\begin{proof}
(i) :  Consider the sequence \ref{exseq:globalideal2}. Any quotient of $\Gamma(I_P(2)\cdot\mathcal{O}_{\mathbb{P}^n_k})$ yields a quotient of $\Gamma(\mathcal{O}_{\mathbb{P}^n_k}(2))$ which has $\Gamma(\mathcal{O}_{\mathbb{P}^n_k}(2))/\Gamma(I_P(2)\cdot\mathcal{O}_{\mathbb{P}^n_k})\cong k$ as quotient, and conversely. This is exactly statement (i).

(ii) : This follows from the isomorphism $\Gamma(I_P(2))\otimes_VK=\Gamma(\mathcal{O}_{\mathbb{P}^n_T}(2))\otimes_VK$.

(iii) : Let $g\in \Gamma(I_P(2))\otimes_Vk$ be a generator of the image of the map $k(P)\to\Gamma(I_P(2))\otimes_Vk$ of sequence \ref{exseq:globalideal1}, and let $\mathcal{Q}'$ be the tautological quotient on $G'_k$. Consider the commutative diagram with exact rows
\[ \xymatrix{
0 \ar[r] & \mathcal{O}_{G'_k}g \ar[r] \ar[d] & \Gamma(I_P(2))\otimes_V\mathcal{O}_{G'_k} \ar[d] \\
0 \ar[r] & \mathcal{L} \ar[r] & \mathcal{Q}'
} \]
where $\mathcal{L}\subset \mathcal{Q}'$ is the image of $\mathcal{O}_{G'_k}g$. Note that $\mathcal{L}\neq 0$ since the image of $g$ is non-zero in a generic quotient of $\Gamma(I_P(2))\otimes_Vk$. Since $G'_k$ is integral and $\mathcal{L}$ is a torsion-free quotient of $\mathcal{O}_{G'_k}g$, it follows easily that $\mathcal{L}=\mathcal{O}_{G'_k}g$. Thus, we get a commutative diagram with exact rows
\[ \xymatrix{
0 \ar[r] & \mathcal{O}_{G'_k}g \ar[r] \ar@{=}[d] & \Gamma(I_P(2))\otimes_V\mathcal{O}_{G'_k} \ar[r] \ar[d] & \Gamma(I_P(2)\cdot\mathcal{O}_{\mathbb{P}^n_k})\otimes_k\mathcal{O}_{G'_k} \ar[r] \ar[d] & 0 \\
0 \ar[r] & \mathcal{O}_{G'_k}g \ar[r] & \mathcal{Q}' \ar[r] & \mathcal{Q}'' \ar[r] & 0
} \]
and there is a dense open $G'_0\subset G'_k$ over which the quotient $\mathcal{Q}''$ is locally free and therefore defines a map $\rho:G'_0\to G_{\mathbb{P}(\Gamma(I_P(2)\cdot\mathcal{O}_{\mathbb{P}^n_k})),r}=\mathcal{H}_P$. Any choice of splitting of the exact sequence \ref{exseq:globalideal1} provides a section of this map, so this map is surjective and this proves (iii).

(iv) : Let $K_L$ denote the residue field of $L$ and $V_L$ the normalization of $V$ in $K_L$. Suppose $L$ is given by a subspace $S\subset\Gamma(I_P(2))\otimes_VK_L$. We claim that the condition $\spe_{G'}(L)\in G'_0$ is equivalent to $S\cap(\Gamma(\mathcal{O}_{\mathbb{P}^n_T}(2))\otimes_VV_L)=S\cap(\Gamma(I_P(2))\otimes_VV_L)$. To see this, setting $Q:=\Gamma(\mathcal{O}_{\mathbb{P}^n_T}(2))\otimes_VV_L/(S\cap \Gamma(\mathcal{O}_{\mathbb{P}^n_T}(2))\otimes_VV_L)$ and $Q':=\Gamma(I_P(2))\otimes_VV_L/(S\cap\Gamma(I_P(2))\otimes_VV_L)$ we have a commutative diagram with exact rows
\[ \xymatrix{
0 \ar[r] & \Gamma(I_P(2))\otimes_VV_L \ar[r] \ar[d] & \Gamma(\mathcal{O}_{\mathbb{P}^n_T}(2))_VV_L \ar[r] \ar[d] & k(P)\otimes_VV_L \ar[d] \ar[r] & 0 \\
0 \ar[r] & Q' \ar[r] & Q \ar[r] & k(P)\otimes_VV_L/M \ar[r] & 0
} \]
where $M=S\cap(\Gamma(\mathcal{O}_{\mathbb{P}^n_T}(2))\otimes_VV_L)/S\cap(\Gamma(I_P(2))\otimes_VV_L)$. Taking tensor product $\otimes_Vk$ we get a commutative diagram with exact rows
\[ \xymatrix{
0 \ar[r] & k(P)\otimes_VV_L \ar[r] \ar[d] & \Gamma(I_P(2))\otimes_Vk\otimes_VV_L  \ar[d] \\
0 \ar[r] & k(P)\otimes_VV_L/M \ar[r] & Q'\otimes_Vk
} \]
In particular, the map $k(P)\otimes_VV_L\to Q\otimes_Vk$ is injective if and only if $M=0$. By definition of $G'_0$ (see the proof of (iii)), this map is injective if and only if the point $l\in G'(k\otimes_VV_L)$ determined by $Q'\otimes_Vk$ lies $G'_0$. Since $\spe_{G'}(L)$ is equal to the support of $l$, we see that $M=0$ if and only if $\spe_{G'}(L)\in G'_0$ as claimed. In particular, since we clearly have $\spe_{G}(\iota^{-1}(L))\in\mathcal{H}_P$ if and only if $M\neq k(P)\otimes_VV_L$, $M=0$ implies $\spe_{G}(\iota^{-1}(L))\in\mathcal{H}_P$ i.e. (iv).

(v) : With notation as in the proof of (iv), note that in the case $K_L=K$ we have $M=0$ if and only if $M\neq k(P)$, hence $\spe_{G'}(L)\in G_0'$ if and only if $\spe_{G}(\iota^{-1}(L))\in\mathcal{H}_P$ as required.
\end{proof}

\begin{lemma}\label{exceptionaldivisorintersection}
With notation as in \ref{closedimmersion} and \ref{zariskiopen}. Assume further that $X$ is regular of dimension $d+1$ at $P$. If $d\geq r$, then there is a dense open $U\subset\mathcal{H}_P$ such that for each $L\in G(T)$ with $\spe_G(L)\in U$, $X\cap L$ is regular at $P$ of dimension $d+1-r$.
\end{lemma}
\begin{proof}
Let $\mathfrak{m}\subset\mathcal{O}_{X}$ be the ideal sheaf of $P$. The closed immersion $E\hookrightarrow\mathbb{P}(\Gamma(I_P(2)))$ of \ref{closedimmersion} is by definition that arising from the line bundle quotient
\[ \Gamma(I_P(2))\otimes_V\mathcal{O}_{E}\to(\mathfrak{m}(2)\cdot\mathcal{O}_{\tilde{X}})\otimes_{\mathcal{O}_{\tilde{X}}}\mathcal{O}_E. \]
Since $E$ is a linear subvariety of $\mathbb{P}(\Gamma(I_P(2)))\otimes_Vk$ under this embedding, it follows that the induced map
\[ \sigma:\Gamma(I_P(2))\to\Gamma(E,\mathfrak{m}(2)\cdot\mathcal{O}_{\tilde{X}}) \]
is surjective and the closed immersion $E\hookrightarrow\mathbb{P}(\Gamma(I_P(2)))$ is equal to the morphism $\Proj(\Sym_V\sigma)$.

Let $S'=\ker(\sigma\otimes_Vk)$ and $S:=\im(S\subset\Gamma(I_P(2))\otimes_Vk\to\Gamma(I_P(2)\cdot\mathcal{O}_{\mathbb{P}^n_k}))$. Let $N=\dim\Gamma(I_P(2))\otimes_Vk$. Note that $\dim S'=N-(d+1)<N-r$, hence $\dim S\leq N-1-r=\dim\Gamma(I_P(2)\cdot\mathcal{O}_{\mathbb{P}^n_k})-r$. So if $\mathcal{Q}''$ denotes the tautological quotient on $\mathcal{H}_P$ and $\mathcal{Q}''':=\im(S\otimes_k\mathcal{O}_{\mathcal{H}_P}\to\mathcal{Q}'')$, then by \ref{subspaceclosed} there is a dense open $U\subset \mathcal{H}_P$ such that $S\otimes_k\mathcal{O}_{U}=\mathcal{Q}'''|_{U}$ and $\mathcal{Q}''/\mathcal{Q}'''|_U$ is flat. Now let $l\in G'_0$ be a closed point such that $\rho(l)\in U$. If $Q'$ denotes the quotient corresponding to $l$ and $Q''$ that corresponding to $\rho(l)$, then as in the proof of \ref{zariskiopen} we have a commutative diagram with exact rows
\[ \xymatrix{
0 \ar[r] & k(P) \ar[r] \ar@{=}[d] & \Gamma(I_P(2))\otimes_Vk \ar[r]^{p} \ar[d]^{q'} & \Gamma(I_P(2)\cdot\mathcal{O}_{\mathbb{P}^n_k}) \ar[r] \ar[d]^{q''} & 0 \\
0 \ar[r] & k(P) \ar[r] & Q' \ar[r] & Q'' \ar[r] & 0
} \]
Since $S\to Q''$ is injective by choice of $U$, the above diagram implies that the map $S'\to Q'$ is injective. Hence $\ker(q')$ injects into $\Gamma(E,\mathfrak{m}(2)\cdot\mathcal{O}_{\tilde{X}})$ under $\sigma\otimes k$. Since $\Gamma(E,\mathfrak{m}(2)\cdot\mathcal{O}_{\tilde{X}})=\mathfrak{m}/\mathfrak{m}^2$, this means that $X\cap L$ is regular at $P$ if $\spe(L)\in U$.
\end{proof}

Now with the picture clarified we can show the following. For a closed point $L\in G_K$ we denote by $K_L$ its residue field and by $V_L$ the normalization of $V$ in $K_L$.

\begin{proposition}\label{smoothcase}
Let $X$ be a smooth quasi-projective $T$-scheme of relative dimension $d\geq r$, $X\hookrightarrow\mathbb{P}^n_T$ an immersion, and let $G:=G_{\mathbb{P}(\Gamma(\mathcal{O}_{\mathbb{P}^n_T}(2))),r}$. Let $P\in X_k$. There is a dense open $U\subset \mathcal{H}_P$ such that for all closed $L\in G_K$ with $\spe_G(L)\in \mathcal{H}_P$ the scheme $X_{V_L}\cap L$ is smooth of relative dimension $d-r$ over $V_L$ at the points of $X_k$.
\end{proposition}
\begin{proof}
Recall that by \ref{zariskiopen} (i) a hyperplane in $\mathbb{P}(\Gamma(I_P(2)\cdot\mathcal{O}_{\mathbb{P}^n_k}))$ corresponds to a hyperplane in $\mathbb{P}(\Gamma(\mathcal{O}_{\mathbb{P}^n_k}(2)))$ through the point $P$. By the Bertini theorem for smoothness (\cite[I, 6.11]{bertini}), the intersection of $r$ generic hyperplanes in $\mathbb{P}(\Gamma(I_P(2)\cdot\mathcal{O}_{\mathbb{P}^n_k}))$ with the image of $X_k-\{P\}$ is smooth of dimension $d-r$. Thus, there is a dense open $U'\subset \mathcal{H}_P$ such that for any $l\in U'$, the intersection $(X_k-\{P\})\cap l$ is smooth of dimension $d-r$. Now let $\mathfrak{m}\subset\mathcal{O}_{X_k,P}$ be the maximal ideal. Fix homogeneous coordinates $y_0,...,y_N$ on $\mathbb{P}(\Gamma(\mathcal{O}_{\mathbb{P}^n_k}(2)))$ such that $P=(1:0:\cdots :0)$. There is a natural surjective $k$-linear map
\begin{eqnarray*}
\Gamma(I_P(2)\cdot\mathcal{O}_{\mathbb{P}^n_k}) &\to &\mathfrak{m}/\mathfrak{m}^2 \\
\sum_{i=1}^Na_iy_i &\mapsto & \sum_{i=1}^Na_i\frac{y_i}{y_0}.
\end{eqnarray*}
Let $S$ be its kernel. If $\mathcal{Q}$ denotes the tautological quotient on $\mathcal{H}_P$ and $\mathcal{Q}':=\im(S\otimes_k\mathcal{O}_{\mathcal{H}_P}\to\mathcal{Q})$, then by \ref{subspaceclosed} there is a dense open $U''\subset \mathcal{H}_P$ such that $S\otimes_k\mathcal{O}_{U''}=\mathcal{Q}'|_{U''}$ and $\mathcal{Q}/\mathcal{Q}'|_{U''}$ is flat. Thus, we have
\[ \ker(\Gamma(I_P(2)\cdot\mathcal{O}_{\mathbb{P}^n_k})\otimes_k\mathcal{O}_{U''}\to\mathcal{Q}|_{U''})=\ker(\mathfrak{m}/\mathfrak{m}^2\otimes_k\mathcal{O}_{U''}\to\mathcal{Q}/\mathcal{Q}'|_{U''}) \]
and so for every $l\in U''$ the corresponding non-degenerate homogeneous linear polynomials $l_1,...,l_r$ generate an $r$-dimensional subspace of $\mathfrak{m}/\mathfrak{m}^2$. Since $P\in X_k$ is a regular point, this implies that $X_k\cap l$ is regular of dimension $d-r$ at $P$. Hence, $X_k\cap l$ is smooth of dimension $d-r$ for all $l\in U:=U'\cap U''$. Now let $L\in G_K$ be a closed point whose specialization lies in $U$. To complete the proof it suffices to show that $X_{V_L}\cap L$ is flat over $V_L$, which follows from the next lemma.
\end{proof}

\begin{lemma}\label{CMflat}
Let $X\to\mathbb{P}^n_T$ be a morphism of finite type, $G=G_{\mathbb{P}(\Gamma(\mathcal{O}_{\mathbb{P}^n_T}(1))),r}$, $L\in G(T)$. Suppose $P\in X$ is a point of $(X\times_{\mathbb{P}^n_T}L)_k$ such that $\dim_PX=d+1$ and $\dim_P(X\times_{\mathbb{P}^n_T}L)_k=d-r$. If $X$ is Cohen-Macaulay at $P$, then $X\times_{\mathbb{P}^n_T}L$ is flat over $T$ at $P$.
\end{lemma}
\begin{proof}
Since the ideal of $X\times_{\mathbb{P}^n_T}L$ in $X$ can be generated by $r$ elements (the hyperplanes corresponding to $L$), we have $\dim_PX\times_{\mathbb{P}^n_T}L\geq \dim_PX-r=d+1-r$. On the other hand, we have $\dim_PX\times_{\mathbb{P}^n_T}L\leq 1+\dim_P(X\times_{\mathbb{P}^n_T}L)_k=1+d-r$. Hence $\dim_PX\times_{\mathbb{P}^n_T}L=d+1-r$ and so $X\times_{\mathbb{P}^n_T}L$ is Cohen-Macaulay since $X$ is. As $\dim_P(X\times_{\mathbb{P}^n_T}L)_k=\dim_PX\times_{\mathbb{P}^n_T}L-1$, it follows that $\pi$ is not a zero divisor on $\mathcal{O}_{X\times_{\mathbb{P}^n_T}L,P}$, i.e. $X\times_{\mathbb{P}^n_T}L$ is flat over $T$ at $P$.
\end{proof}

\begin{proposition}\label{awayfromP}
Let $X$ be a flat projective $T$-scheme of dimension $d+1$ and with geometrically reduced generic fibre $X_K$. Let $X\hookrightarrow\mathbb{P}^n_T$ be a closed immersion, and $G:=G_{\mathbb{P}(\Gamma(\mathcal{O}_{\mathbb{P}^n_T}(2))),d}$. Fix a closed point $P\in\mathbb{P}^n_k$. Then there is a dense open $U\subset \mathcal{H}_P$ such that for any closed $L\in G_K$ such that $\spe_G(L)\in U$, $X_k\cap\spe_G(L)$ is finite over $k$ and $X_{K_L}\cap L$ is \'etale over $K_L$ away from $\spe^{-1}(P)\cap X_{K_L}\cap L$. In particular, if $P\notin X$, then $X_{K_L}\cap L$ is \'etale over $K_L$.
\end{proposition}
\begin{proof}
Let $X^{\sm}\subset X$ be the open subset of points smooth of relative dimension $d$ over $T$. We first show that $\dim(X-X^{\sm})\leq d$. If not, then there is an irreducible component $Y$ of $X$ with $\dim Y=d+1$ and $Y\cap X^{\sm}=\emptyset$. Since $Y_K$ is geometrically reduced this is only possible if $Y_K=\emptyset$, which cannot happen as $\dim X_k=d$ ($X$ is flat over $T$). Thus, $\dim(X-X^{\sm})\leq d$.

Now there is  $\Omega^1_{X/T}|_{X^{\sm}}$ is locally free of rank $d$ and so defines a morphism $X^{\sm}\to\Grass_d(\Omega^1_{X/T})$ of $X$-schemes, which is an open immersion since $\Grass_d(\Omega^1_{X/T})\times_XX^{\sm}=\Grass_d(\Omega^1_{X^{\sm}/T})=X^{\sm}$. Let $\mathcal{X}$ be the scheme-theoretic closure of $X^{\sm}$ in $\Grass_d(\Omega^1_{X/T})$, and let $m:\mathcal{X}\to X$ be the canonical map. Since $\mathcal{X}$ is reduced (because $X^{\sm}$ is) and each of its irreducible components dominates $T$, $\mathcal{X}$ is flat over $T$ and so $\dim\mathcal{X}_k=\dim\mathcal{X}_K$ (\cite[IV, 14.2.5]{ega}), whence $\dim\mathcal{X}_k=d$. Moreover, by construction there is a vector bundle $\mathcal{E}$ of rank $d$ on $\mathcal{X}$ and a surjective map
\[ m^*\Omega^1_{X/T}\twoheadrightarrow\mathcal{E} \]
of $\mathcal{O}_{\mathcal{X}}$-modules. Since $\mathring{X}:=X_k-\{P\}$ is contained in $\mathbb{P}(\Gamma(I_P(2)\cdot\mathcal{O}_{\mathbb{P}^n_k}))$ we can apply \ref{bertini7} to $\mathring{\mathcal{X}}:=\mathcal{X}_k-m^{-1}(P)$ to find a dense open $U\subset\mathcal{H}_P$ such that for all closed $l\in U$, if $I$ denotes the ideal sheaf of $\mathring{X}$ generated by non-degenerate hyperplanes $l_1,...,l_d$ corresponding to $l$, then the map
\[ I/I^2\otimes_{\mathcal{O}_{\mathring{X}}}\mathcal{O}_{\mathring{\mathcal{X}}}\to\mathcal{E}|_{\mathring{\mathcal{X}}\times_{X_k}(X_k\cap l)} \]
is surjective. So if $L\in G_K$ is a closed point such that $\spe(L)\in U$ and $I_L\subset\mathcal{O}_{X_{V_L}}$ denotes the ideal generated by non-degenerate hyperplanes corresponding to $L$ then by Nakayama's lemma the map
\[ I_L/I_L^2\otimes_{\mathcal{O}_{X}}\mathcal{O}_{\mathcal{X}}\to\mathcal{E}|_{\mathcal{X}\times_X(X_{V_L}\cap L)} \]
is surjective at the points of $X^{\sm}_{K_L}\cap L$ which specialize to points of $\mathring{X}$. This implies that $X^{\sm}_{K_L}\cap L$ is \'etale over $K$ at the points which do not specialize to $P$.

To finish, we claim that we can shrink $U$ so that $Z_{K_L}\cap L\subset\spe^{-1}_X(P)$, where $Z:=X-X^{\sm}$. To see this, note that $\dim Z\leq d$, so if $C\subset Z$ is an irreducible component which dominates $T$, then $\dim C_k<d$. Hence up to shrinking $U$ we can assume $(C_k-\{P\})\cap l=\emptyset$ for all $l\in U$, whence $C_{K_L}\cap L\subset\spe^{-1}_X(P)$ for all $L\in G_K$ with $\spe(L)\in U$. If $C\subset Z$ is an irreducible component which doesn't dominate $T$, then $C_K=\emptyset$. This proves the claim.
\end{proof}

We now prove Theorem \ref{goodetalemap}. We fix $G:=G_{\mathbb{P}(\Gamma(\mathcal{O}_{\mathbb{P}^n_T}(2))),d}$.

\begin{theorem}\label{maingoodnhbprop}
Assume $\cha(K)=0$. Let $X$ be an flat projective $T$-scheme of dimension $d+1$ and with reduced generic fibre $X_K$. Let $X\hookrightarrow\mathbb{P}^n_T$ be a closed immersion and fix a closed point $P\in X_k$ with $\dim_PX=d+1$.
\begin{enumerate}[(i)]
\item If $X$ is regular at $P$, then there is a dense open $U\subset \mathcal{H}_P$ such that for any $L\in G(K)$ with $\spe_{G}(L)\in U$, $X_k\cap\spe_{G}(L)$ is finite over $k$ and $X_{K}\cap L$ is \'etale over $K$.
\item If $X$ is smooth at $P$, then there is a dense open $U\subset \mathcal{H}_P$ such that for any closed $L\in G_K$ with $\spe_{G}(L)\in U$, $X_k\cap\spe_{G}(L)$ is finite over $k$ and $X_{K_L}\cap L$ is \'etale over $K_L$.
\end{enumerate}
Furthermore, we have $X_{V_L}\cap L\neq\emptyset$ for all such $L$.
\end{theorem}
\begin{proof}
(i) : Since $X$ is regular, by \ref{exceptionaldivisorintersection} there is a dense open $U'\subset\mathcal{H}_P$ such that, for all $L\in G(T)$ with $\spe_G(L)\in U'$, $X\cap L$ is regular at $P$, and flat over $T$ at $P$ of dimension 1 by \ref{CMflat}. Thus, $X_K\cap L$ is \'etale over $K$ at the points specializing to $P$. Moreover, since $\dim X_k=d$, up to shrinking $U'$ we may assume that for all $l\in U'$ we have $\dim (X_k-\{P\})\cap l=0$, hence $X_k\cap l$ is finite. For the remaining points, by \ref{awayfromP} there is a dense open $U''\subset\mathcal{H}_P$ such that for any $L\in G(K)$ with $\spe(L)\in U''$ the scheme $X_K\cap L$ is \'etale at the points which not do specialize to $P$.

Hence we can take $U=U'\cap U''$ for the first assertion. For the assertion that $X_{K}\cap L\neq\emptyset$, consider the local ring $\mathcal{O}_{X,P}$. By assumption it has dimension $d+1$. Since $L$ is the intersection of $d$ hyperplanes passing through $P\in X$, we have $\mathcal{O}_{X\cap L,P}\neq 0$ and hence $\dim\mathcal{O}_{X\cap L,P}\geq 1$. If $\mathcal{O}_{X\cap L,P}\otimes_VK=0$, then the kernel of the canonical map $\mathcal{O}_{X,P}\to\mathcal{O}_{X\cap L,P}$ contains a power of the uniformizer $\pi\in V$. Thus, $\dim\mathcal{O}_{X\cap L,P}=\dim\mathcal{O}_{X_k\cap L,P}=0$, a contradiction. Hence $\mathcal{O}_{X\cap L,P}\otimes_VK\neq 0$, completing the proof.

(ii) : The proof is similar to (i), the reference to \ref{exceptionaldivisorintersection} being replaced by \ref{smoothcase}.
\end{proof}

\section{Construction of good neighbourhoods}
In this section we assume $\cha(K)=0$.

\subsection{Setup}
Let $\bar{X}$ be a connected smooth projective $T$-scheme of finite type of relative dimension $d$. Fix a closed immersion $\bar{X}\hookrightarrow\mathbb{P}^n_T$, an open subset $X\subset\bar{X}$, and a closed point $P\in X_k$.

\subsection{Good hyperplanes}\label{goodnbhdsetup}
Let $V\subset V'$ be a finite extension of discrete valuation rings, $K'=V'\otimes_VK$, and $T':=\Spec(V')$. Let $H_0,H_1,...,H_{d-1}$ be a set of $T'$-hyperplanes and $L:=H_1\cap H_2\cap\cdots\cap H_{d-1}$. Consider the following conditions:
\begin{enumerate}
\item $H_1,...,H_{d-1}$ are non-degenerate and pass through $P$, and $H_0$ does not pass through $P$
\item $\bar{X}_{T'}\cap L$ is smooth of relative dimension 1 over $T'$
\item $Y_{K'}\cap L$ is \'etale over $K'$
\item $\bar{X}_{T'}\cap L\cap H_0$ is \'etale over $T'$
\item if $D$ is the scheme-theoretic closure of $Y_K$ in $Y$, then $(D_{T'}\cap L\cap H_0)_k=\emptyset$.
\suspend{enumerate}
Suppose $H_0,....,H_{d-1}$ are $T'$-hyperplanes satisfying (1)-(5). Note that since $\bar{X}_{K}$ is geometrically connected, by (2) it follows from \cite[I, 7.1]{bertini} that $\bar{X}_{K'}\cap L$ is a smooth geometrically connected projective $K'$-curve. Moreover, (4) and (5) imply that $X_{K'}\cap L\cap H_0=\bar{X}_{K'}\cap L\cap H_0$ is finite \'etale over $K$.

We say that $T$-hyperplanes $H_0,...,H_{d-1}$ are \emph{good} if they satisfy conditions (1)-(5) and moreover
\resume{enumerate}
\item if, for any finite extension $T'\to T$ as above, $H_1',...,H_{d-1}'$ is any set of non-degenerate $T'$-hyperplanes with same specialization (as a point of the appropriate grassmannian) as $H_1,...,H_{d-1}$, then conditions (2)-(5) hold for the hyperplanes $(H_0)_{T'},H_1',...,H_{d-1}'$.
\end{enumerate}

\begin{proposition}\label{good}
Good $T$-hyperplanes exist, up to replacing the closed immersion $\bar{X}\hookrightarrow\mathbb{P}^n_T$ by its double.
\end{proposition}
\begin{proof}
Let $G:=G_{\mathbb{P}(\Gamma(\mathcal{O}_{\mathbb{P}^n_T}(2))),d-1}$ (notation as in \S\ref{setupgrass}). By \ref{smoothcase} (resp. \ref{awayfromP}), there is a dense open $U_1\subset \mathcal{H}_P$ such that condition (2) (resp. (3)) holds for any closed $L\in G_K$ with $\spe(L)\in U_1$.

Pick $z\in U(k)$. Since $G$ is smooth over $T$ we can lift $z\in U_k$ to a point in $G(T)$ and this gives $H_1,...,H_{d-1}$.

By the Bertini theorem for smoothness one can find $H_0$ such that $X_k\cap L\cap H_0$ is \'etale over $k$. Then (4) holds by \ref{CMflat}. Moreover, since $\dim(D_k\cap\spe(L))=0$, we can choose $H_{0}$ such that it does not intersect $D_k\cap\spe(L)$, hence (5) is satisfied. So $H_0,H_1,...,H_{d-1}$ satisfy (1)-(5).

Finally, for condition (6), note that it holds for (5) automatically since it only depends on the special fibre of $L$ and $H_0$, and that it also holds for (2), (3), and (4) given our choice of $U$ and $H_0$.
\end{proof}

\subsection{The map to affine space} Let $H_0,H_1,...,H_{d-1}$ be good $T$-hyperplanes and let $F_i$ be a homogeneous equation of $H_i$ for $i=0,1,...,d-1$. Let $Y_1,...,Y_{d-1}$ be coordinates on $\mathbb{A}^{d-1}_T$ and let
\[ \bar{X}'=V(F_1-Y_1F_0,F_2-Y_2F_0,...,F_{d-1}-Y_{d-1}F_0) \]
be the closed subscheme of $\bar{X}\times_T\mathbb{A}^{d-1}_T$ of homogeneous ideal generated by the polynomials $F_1-Y_1F_0,F_2-Y_2F_0,...,F_{d-1}-Y_{d-1}F_0$. Let $p:\bar{X}'\to\mathbb{A}^{d-1}_T$, and $q:\bar{X}'\to\bar{X}$ be the natural projections. Note that $p$ is a projective morphism by construction. Finally, let $C:=L\cap H_0$.

\begin{lemma}\phantomsection\label{loc}
\begin{enumerate}[(i)]
\item For $z\in \bar{X}\cap C$ we have $q^{-1}(z)=\mathbb{A}^{d-1}_{k(z)}$, where $k(z)$ denotes the residue field of $\bar{X}$ at $z$.
\item The map $\bar{X}'\setminus V(F_0)\to \bar{X}\setminus V(F_0)$ induced by $q$ is an isomorphism.
\item If we define $Z:=V(F_0)\subset\bar{X}'$, then $Z=(\bar{X}\cap C)\times_T\mathbb{A}^{d-1}_T$.
\end{enumerate}
\end{lemma}
\begin{proof}
(i) : $q^{-1}(z)$ is the closed subscheme of $\mathbb{A}^{d-1}_{k(z)}$ cut out by the polynomials $F_1(z)-Y_1F_0(z),F_2(z)-Y_2F_0(z),...,F_{d-1}(z)-Y_{d-1}F_0(z)$. Since $z\in\bar{X}\cap C$, $F_i(z)=0$ for $0\leq i\leq d-1$, and the claim follows.
(ii) : $U:=\bar{X}\setminus V(F_0)$ is affine and $\bar{X}'\setminus V(F_0)$ is the closed subscheme of $U\times\mathbb{A}^{d-1}$ of ideal generated by $F_1/F_0-Y_1,...,F_{d-1}/F_0-Y_{d-1}$, which is obviously isomorphic to $U$.

(iii) : $Z$ is the closed subscheme of $\bar{X}\times\mathbb{A}^{d-1}$ of ideal generated by $F_0,F_1,..,F_{d-1}$, i.e. $(\bar{X}\cap C)\times\mathbb{A}^{d-1}$.
\end{proof}

Define $X':=\bar{X}-(Y\cup (\bar{X}\cap H_0))=X-X\cap H_0$. Note that $P\in X'$ since $P\notin H_0$. Moreover, by \ref{loc} (ii) the map $q$ induces an isomorphism $q^{-1}(X')\cong X'$ so we shall identity $X'$ with $q^{-1}(X')$. Finally set $Y':=\bar{X}'-X'$.

Now, the fibre over the origin $0\in\mathbb{A}^{d-1}(T)$ of $p:\bar{X}'\to\mathbb{A}^{d-1}_T$ is the scheme $\bar{X}\cap L$.

\begin{lemma}\label{spread}
Let $y\in\mathbb{A}^{d-1}_{K}$ be a closed point specializing to $0\in\mathbb{A}^{d-1}_k$.
\begin{enumerate}[(i)]
\item $p^{-1}_K(y)$ is a smooth and geometrically connected curve over $y$.
\item $p^{-1}_K(y)\cap Y$ is \'etale over $y$.
\item There is an open neighbourhood of $y$ in $\mathbb{A}^{d-1}_K$ such that the restriction of $p$ to $Y'$ is finite \'etale over this neighbourhood.
\item There is an open neighbourhood of $y$ in $\mathbb{A}^{d-1}_K$ such that $p$ is smooth over this neighbourhood.
\end{enumerate}
\end{lemma}
\begin{proof}
(i) and (ii) : Let $V'$ be the normalization of $V$ in the residue field $K'$ at $y$, $T'=\Spec(V')$, and let $\pi'\in V'$ be a uniformizer. Since $\spe(y)=\spe(0)$, $y$ extends to a point $(y_1,...,y_{d-1})\in\mathbb{A}^{d-1}(T')$ with $y_i\equiv 0\mod\pi'$. So the fibre over $(y_1,...,y_{d-1})$ is the intersection of $\bar{X}$ with the hyperplanes $H_1',...,H_{d-1}'$ defined by the homogeneous polynomials $F_1':=F_1-y_1F_0,F_2':=F_2-y_2F_0,...,F_{d-1}':=F_{d-1}-y_{d-1}F_0$ respectively. We claim that $H_1',...,H_{d-1}'$ are non-degenerate. Since $F_i'\equiv F_i\mod\pi'$, it follows that $\wedge_{i=1}^{d-1}F_i'\equiv \wedge_{i=1}^{d-1}F_i\mod \pi'$, hence $\wedge_{i=1}^{d-1}F_i'$ forms part of a basis of $\wedge^{d-1}\mathcal{O}_{T'}^{n+1}$, i.e. the $F_i'$ are non-degenerate. Let $L':=\cap_iH_i'$. Since $H_0,...,H_{d-1}$ are good and $H_1',...,H_{d-1}'$ are non-degenerate hyperplanes with same specialization as $H_1,...,H_{d-1}$, (i) and (ii) hold by definition.

(iii) : By \ref{loc} (iii) we have $Z_K=(\bar{X}\cap C)_K\times_K\mathbb{A}^{d-1}_K$, and since $H_0,...,H_{d-1}$ are good we have $(\bar{X}\cap C)_K=(X\cap C)_K$ is finite \'etale over $K$. Thus $Z_K$ is finite \'etale over $\mathbb{A}^{d-1}_K$. Now, by (ii) $(Y_{T'}\cap L')_{K'}$ is finite \'etale over $K$. Also, since $H_0,H_1,...,H_{d-1}$ are good, for every irreducible component $E\subset Y$ dominating $T$ we have $(E\cap L'\cap H_0)_k=\emptyset$, so by properness we deduce that $E_{T'}\cap L'\cap H_0=\emptyset$. Hence $(Y_{T'}\cap L'\cap H_0)_{K'}=\emptyset$, and $Y'_{T'}\cap p^{-1}_K(y)$ is the disjoint union of $p^{-1}_K(y)\cap Z_K$ and $Y_{K'}\cap L'$, both of which are finite \'etale over $k(y)$ by the above. So the morphism $Y'_K\to\mathbb{A}^{d-1}_K$ is dominant, proper, and its fibre over $y$ is \'etale. Since $\mathbb{A}^{d-1}_K$ is normal it follows from \cite[IV, 18.10.3]{ega} that this morphism is finite \'etale in a neighbourhood of $y$.

(iv): We claim that $p_K$ is dominant. Since $p_K$ is proper, $S:=p_K(\bar{X}'_K)\subset\mathbb{A}^d_K$ is closed, so is equal to $\mathbb{A}^d_K$ if and only if $\dim S=d$. Consider the $K$-analytic space $S^{\an}$ associated to $S$ by Berkovich \cite[3.4]{berkovich}. It is a closed analytic subspace of $(\mathbb{A}^d_K)^{\an}$. Now, to show that $\dim S=d$ it suffices to show that $W\subset S^{\an}$ for some analytic open $W\subset(\mathbb{A}^d_K)^{\an}$ (for then at any point $y\in W$ whose residue field is finite over $K$, the Krull dimension of $\mathcal{O}_{W,y}$ is $d$ and $\widehat{\mathcal{O}}_{W,y}=\widehat{\mathcal{O}}_{S,y}$ \cite[3.4.1]{berkovich}). Let $B$ be the closed unit ball around the origin in $(\mathbb{A}^d_K)^{\an}$, i.e. $B=\mathcal{M}(K\{Y_1,...,Y_d\})$ in the notation of \cite{berkovich}. The reduction map $r:B\to\mathbb{A}^d_k$ is anti-continuous and surjective (\cite[2.4]{berkovich}) and one sees easily that $\mathring{B}:=r^{-1}(0)$ is the open unit ball around the origin in $(\mathbb{A}^d_K)^{\an}$. By the above we know that $S$ contains the set $\spe^{-1}_{\mathbb{A}^d_T}(0)$ ($0\in\mathbb{A}^{d}_k$). Hence, it follows from the definitions that $S^{\an}$ contains the set $\Sigma$ of points of $\mathring{B}$ whose residue field is finite over $K$. It follows from \cite[2.1.15]{berkovich} that $\Sigma$ is everywhere dense in $\mathring{B}$. Thus, $S^{\an}\cap\mathring{B}$ is a closed everywhere dense subset of $\mathring{B}$, whence $\mathring{B}\subset S^{\an}$, as was to be shown.

Now, up to making a finite extension we may assume $y\in\mathbb{A}^{d-1}(K)$. By \cite[Exp. II, 2.6]{sga1} it suffices to see that in the generic point $\eta\in p^{-1}_K(y)$ we have $\dim(\mathcal{O}_{\bar{X}',\eta})=d-1$ and that the maximal ideal of  $\mathcal{O}_{\bar{X}',\eta}$ is generated by that of $\mathcal{O}_{\mathbb{A}^{d-1},y}$. Now, we claim that $X'_K\cap L'$ is dense in $\bar{X}'_K\cap L'$. First note that $X_K\cap L'$ is dense in $\bar{X}_K\cap L'$ since the latter is irreducible and $Y_K\cap L'$ is finite. So $\dim(X_{K}\cap L')=1$ and since $(X\cap C)_{K}$ is finite it follows that $X'_{K}\cap L'=X_{K}\cap L'-(X\cap C)_{K}\neq\emptyset$, so $\dim(X'_{K}\cap L')=1$. Thus, $X'_{K}\cap L'=\bar{X}'_K\cap L'-Y'_{K}\cap L'$ is dense in $\bar{X}'_{K}\cap L'$ and we have $\mathcal{O}_{\bar{X}',\eta}=\mathcal{O}_{X',\eta}$. Since $X'_{K}$ is regular and its intersection with the hyperplanes $\left\{y_i=F_i/F_0\right\}_{1\leq i\leq d-1}$ is a smooth curve, it follows that $F_1/F_0-y_1,...,F_{d-1}/F_0-y_{d-1}$ generate the maximal ideal of $\mathcal{O}_{X',\eta}$. Finally, we have $F_i/F_0-y_i=Y_i-y_i$ so that $Y_1-y_1,...,Y_{d-1}-y_{d-1}$ do indeed generate the maximal ideal of $\mathcal{O}_{X',\eta}$.
\end{proof}

By the last lemma there exists an open $U_0\subset\mathbb{A}^{d-1}_{K}$, containing the set $\spe^{-1}_{\mathbb{A}^{d-1}_T}(0)$ ($0\in\mathbb{A}^{d-1}_k$), such that $p$ is smooth with geometrically connected fibres over $U_0$ and its restriction to $Y'$ is finite \'etale. It follows from \ref{sp} that $U_0=U_{K}$ for some open neighbourhood $U$ of $0\in\mathbb{A}^{d-1}_k$. So $p|_{p^{-1}(U_K)}$ is an elementary fibration (in the sense of \cite[Exp. XI]{sga4.3}) and by induction on the dimension of $X$ we deduce

\begin{theorem}\label{elementaryfib}
If $X$ is an open in a smooth projective $T$-scheme and $P\in X_k$ a closed point, then there is an open neighbourhood $X''$ of $P$ in $X$ such that $X''_{K}$ is a good neighbourhood relative to $\Spec(K)$ in the sense of \cite[Exp. XI]{sga4.3}.
\end{theorem}
\begin{proof}
We have already established that $p_K^{-1}(U_K)\to U_K$ is an elementary fibration. By induction on $\dim X_K$, there is an open neighbourhood $U'\subset U$ of $0\in\mathbb{A}^{d-1}_k$, such that $U'_K$ is a good neighbourhood relative to $\Spec(K)$ in the sense of \cite[Exp. XI]{sga4.3}. So we can take $X''=X'\cap p^{-1}(U')$, where $X'$ is defined as above.
\end{proof}

\subsection{$K(\Pi,1)$ neighbourhoods}\label{kpi1section}
Let $X$ be a scheme and let $X_{\Fet}$ be the finite-\'etale site of $X$, i.e. $X_{\Fet}$ is the subsite of the \'etale site $X_{\et}$ whose objects are the finite \'etale $X$-schemes. Recall that $X$ is a $K(\Pi,1)$ (for the \'etale topology) if, for each integer $n$ invertible on $X$ and any sheaf $\mathbb{L}$ of $\mathbb{Z}/n$-modules on $X_{\Fet}$, the adjunction map $\mathbb{L}\to R\rho_*\rho^*\mathbb{L}$ is a quasi-isomorphism, where $\rho:X_{\et}\to X_{\Fet}$ is the canonical morphism of sites. Intuitively, for $X$ connected this means that its \'etale cohomology with values in such a sheaf $\mathbb{L}$ can be computed as the Galois cohomology of the \'etale fundamental group of $X$ with values in the stalk of $\mathbb{L}$ at the choice of base point.

Now, since $\cha(K)=0$, a good neighbourhood of a $K$-variety is a $K(\Pi,1)$ (\cite[Exp. XI, 4.6]{sga4.3}; for an algebraic proof using Abhyankar's lemma see \cite[5.5]{olsson}), so \ref{elementaryfib} implies

\begin{corollary}\label{corelementaryfib}
If $X$ is an open in a smooth projective $T$-scheme, any closed point $P\in X_k$ has an open neighbourhood $U$ such that $U_K$ is a $K(\Pi,1)$.
\end{corollary}

\section{Construction of $K(\Pi,1)$ neighbourhoods}
We maintain the assumption $\cha(K)=0$.

\subsection{Setup}\label{setupkpi1}
Let $X$ be a connected affine smooth $T$-scheme of relative dimension $d$. Fix a closed immersion $X\hookrightarrow\mathbb{A}^n_T$ and let $\bar{X}$ be the closure of $X$ in $\mathbb{P}^n_T$, $Y:=\bar{X}-X$. Fix a closed point $P\in X_k$.

\subsection{Nice hyperplanes} Let $T'$ be a finite extension of $T$ as in \ref{goodnbhdsetup}. Let $H_0,H_1,...,H_{d}$ be $T'$-hyperplanes and $L:=H_1\cap H_2\cap\cdots\cap H_{d}$. Consider the following conditions:
\begin{enumerate}
\item $H_1,H_2,..,H_{d}$ are non-degenerate and pass through $P$, and $H_0$ does not pass through $P$
\item if $D$ is the scheme theoretic closure of $Y_K$ in $Y$, then $(D_{T'}\cap L)_k=\emptyset$
\item $(\bar{X}_{T'}\cap L)_k$ is finite over $k$ and $(\bar{X}_{T'}\cap L)_K$ is \'etale over $K$ and non-empty
\item $(\bar{X}_{T'}\cap L\cap H_0)_k=\emptyset$.
\suspend{enumerate}
We say that $T$-hyperplanes $H_0,H_1,...,H_{d-1}$ are \emph{nice} if they satisfy (1)-(4) and moreover
\resume{enumerate}
\item if, for any finite extension $T'\to T$ as above, $H_1',...,H'_{d}$ are $T'$-hyperplanes with same specialization (as a point of the suitable grassmannian) as $H_1,...,H_{d}$, then conditions (2)-(4) hold for the hyperplanes $(H_0)_{T'},H_1',...,H_d'$.
\end{enumerate}

\begin{proposition}
Nice $T$-hyperplanes exist, up to replacing the embedding $\bar{X}\hookrightarrow\mathbb{P}^n_T$ by its double.
\end{proposition}
\begin{proof}
Let $G:=G_{\mathbb{P}(\Gamma(\mathcal{O}_{\mathbb{P}^n_T}(2))),d}$ (\S\ref{setupgrass}).

For (2), note that $\bar{X}$ is integral (since $X$ is), hence flat over $T$ and of dimension $d+1$. Since the space underlying $Y$ is equal to the intersection of $\bar{X}$ with the hyperplane at infinity, it follows that $\dim Y=d$. So $\dim D_k\leq d-1$. As $P\notin D$, there is a dense open $U_1\subset \mathcal{H}_{P}$ such that $D_k\cap l=\emptyset$ for all $l\in U_1$.

By \ref{maingoodnhbprop} (ii), there is a dense open $U_2\subset \mathcal{H}_P$ such that condition (3) holds for all closed $L\in G_K$ with $\spe(L)\in U_2$.

So we can pick a closed point in $U_1\cap U_2$, and it can be lifted to an element of $G(T)$. This point gives the hyperplanes $H_1,...,H_d$ which satisfy conditions (2) and (3).

In view of (3), we can clearly pick $H_0$ such that condition (4) is satisfied. The hyperplanes $H_0,...,H_{d-1}$ satisfy (1)-(4).

Finally, condition (5) holds for (3) by choice of $U_2$ and as (2) and (4) depend only on $\spe(L)$, (5) holds automatically for them.
\end{proof}

\subsection{The map to affine space} Let $H_0,H_1,...,H_{d}$ be nice $T$-hyperplanes. Let $F_i$ be a homogeneous equation of $H_i$ for $i=0,1,...,d$. Let 
\[ \bar{X}'=V(F_1-Y_1F_0,F_2-Y_2F_0,...,F_{d}-Y_{d}F_0) \]
be the closed subscheme of $\bar{X}\times_T\mathbb{A}^{d}_T$ defined by the homogeneous polynomials $F_1-Y_1F_0,F_2-Y_2F_0,...,F_{d}-Y_{d}F_0$, where $Y_1,...,Y_{d}$ are the coordinates on $\mathbb{A}^{d}$. Let $p:\bar{X}'\to\mathbb{A}^{d}_T$, and $q:\bar{X}'\to\bar{X}$ be the natural projections.

Define $X':=\bar{X}-(Y\cup (H_0\cap\bar{X}))$. Note that $P\in X'$. By the same argument as \ref{loc} (ii), $q$ induces an isomorphism $q^{-1}(X')\cong X'$ which we use to identify $X'$ and $q^{-1}(X')$. In fact, since $H_0,...,H_d$ are nice we have $\bar{X}\cap L\cap H_0=\emptyset$ so $V(F_0)\cap\bar{X}'=\emptyset$ and the map $q$ is an open immersion.

\begin{theorem}\label{Kpi1}
There is an open neighbourhood $U\subset\mathbb{A}^d_T$ of $0\in\mathbb{A}^{d}_k$ such that the map $X'_{U_K}\to U_K$ induced by $p$ is finite \'etale with non-empty fibres.
\end{theorem}
\begin{proof}
Let $y\in\mathbb{A}^d_K$ be a closed point specializing to $0\in\mathbb{A}^{d}_k$. As in the proof of of \ref{spread} (i), the fibre $p^{-1}_K(y)$ is the intersection of $\bar{X}_{K(y)}$ with non-degenerate hyperplanes $H_1',...,H_d'$ having the same specialization as $H_1,...,H_d$ (here $K(y)$ is the residue field at $y$). Since $H_1,...,H_d$ are nice, $p^{-1}_K(y)$ is finite \'etale and non-empty by definition.

Note that $p_K$ is dominant: this can be seen as in the proof of \ref{spread} (iv). Now since the morphism $p_K$ is proper, dominant, and $\mathbb{A}^d_K$ is normal, it follows from \cite[IV, 18.10.3]{ega} that there is an open $U'\subset\mathbb{A}^d_K$ containing all points specializing to $0\in\mathbb{A}^{d}_k$ such that $p_K^{-1}(U')\to U'$ is finite \'etale.

Moreover, since $H_0,...,H_d$ are nice we have $Y_{K(y)}\cap p_K^{-1}(y)=\emptyset$ for all $y\in\mathbb{A}^d_K$ with $\spe(y)=0\in\mathbb{A}^d_k$. So if $Y':=\bar{X}'\setminus X'$, then $Y'_{K(y)}\cap p_K^{-1}(y)=\emptyset$. Since $Y'_K$ is proper over $\mathbb{A}^d_K$, up to localizing over $U'$ we may assume that $\bar{X}'_{U'}=X'_{U'}$.

Finally, since $U'$ contains the set $\spe^{-1}_{\mathbb{A}^d_T}(0)$ ($0\in\mathbb{A}^{d}_k$), by \ref{sp} it follows that $U'=U_K$ for some open $U\subset\mathbb{A}^d_T$ containing $0\in\mathbb{A}^{d}_k$, as required.
\end{proof}

\begin{corollary}[Faltings \cite{fa2}]\label{corKpi1}
There is an open neighbourhood $X''$ of $P$ in $X$ such that $X''_K$ is a $K(\Pi,1)$.
\end{corollary}
\begin{proof}
According to \ref{Kpi1} there is an open neighbourhood $X''$ of $P$ and a morphism $f:X''\to U$ with $U$ open in $\mathbb{A}^d_T$ such that $X''_K\to U_K$ is a finite \'etale covering. By \ref{corelementaryfib}, up to shrinking $U$ around $f(P)$ we may assume that $U_K$ is a $K(\Pi,1)$. Since a finite \'etale covering of a $K(\Pi,1)$ is also a $K(\Pi,1)$ this implies the result.
\end{proof}

\subsection{Application to $p$-adic nearby cycles}\label{padic}
Let $X$ be a smooth $T$-scheme. Denote by $\bar{K}$ an algebraic closure of $K$. Recall that Faltings \cite{fa4} has defined a site whose objects are pairs $(U,W)$, where $U\to X$ is \'etale and $W\to U_{\bar{K}}$ is finite \'etale. It has been noticed that the topology on this category given in loc. cit. is not quite correct. This issue has been studied in detail in the work of Abbes and Gros \cite{abbesgros}. The right topology is the following \cite[10]{abbesgros}: it is the topology generated by the coverings $\{(W_i,U_i)\to (W,U)\}_{i\in I}$ such that either $U_i=U$ for all $i\in I$ and $(W_i\to W)_{i\in I}$ is an \'etale covering, or $(U_i\to U)_{i\in I}$ is an \'etale covering and $W_i=W\times_UU_i$. We denote by $\mathcal{X}$ the site of such pairs $(U,W)$ with this topology.

\begin{corollary}
Let $w:X_{\bar{K},\et}\to\mathcal{X}$ be the morphism of sites induced by the functor $(U,W)\mapsto W$. For any locally constant constructible \'etale sheaf of abelian groups $\mathbb{L}$ on $X_{\bar{K}}$, $w$ induces a canonical isomorphism
\[ H^{\ast}(\mathcal{X},w_{\ast}\mathbb{L})\cong H^{\ast}_{\et}(X_{\bar{K}},\mathbb{L}). \]
More precisely, if $\sigma:\mathcal{X}\to X_{\et}$ is the morphism induced by the functor $U\mapsto (U,U_{\bar{K}})$ and $\bar{j}:X_{\bar{K},\et}\to X_{\et}$ the canonical map, then the natural map $R\sigma_{\ast}(w_*\mathbb{L})\to R\bar{j}_{\ast}\mathbb{L}$ is a quasi-isomorphism.
\end{corollary}
\begin{proof}
Let $i\in\mathbb{N}$ and let us show that the canonical map $\psi:R^i\sigma_{\ast}(w_*\mathbb{L})\to R^i\bar{j}_{\ast}\mathbb{L}$ is an isomorphism. Let $\mathcal{H}^i$ be the sheaf on $X_{\et}$ associated to the presheaf
\[ U\mapsto H^i_{\Fet}(U_{\bar{K}},\mathbb{L}) \]
where the subscript $\Fet$ means we take cohomology for the finite-\'etale topology. There is a canonical map (\cite[10.40]{abbesgros})
\[ \phi:\mathcal{H}^i\to R^i\sigma_{\ast}(w_*\mathbb{L}) \]
such that the composition $\psi\circ\phi$ is the map induced by the natural transformation $H^i_{\Fet}((-)_{\bar{K}},\mathbb{L})\to H^i_{\et}((-)_{\bar{K}},\mathbb{L})$. Let us show that $\psi\circ\phi$ is an isomorphism. Let $x\to X$ be a geometric point. Let $U\to X$ be an \'etale neighbourhood of $x$. Choose a specialization $y\to U$ of $x$ such that the image of $y$ is closed in $U$. If $y\in U_k$ (resp. $y\in U_K$), then by \ref{corKpi1} (resp. \cite[Exp. XI]{sga4.3}) there is an open $U'\subset U$ such that $y\to U$ factors through $U'$ with $U'_K$ (hence also $U'_{\bar{K}}$) a $K(\Pi,1)$. Then the map $x\to U$ factors through $U'$. So this implies that the set $I$ of all \'etale neighbourhoods $U$ of $x$ such that $U_{\bar{K}}$ is $K(\Pi,1)$ is cofinal in the set of all \'etale neighbourhoods of $x$. Thus,
\[ \mathcal{H}^i_x=\varinjlim_{U\in I}H^i_{\Fet}(U_{\bar{K}},\mathbb{L})=\varinjlim_{U\in I}H^i_{\et}(U_{\bar{K}},\mathbb{L})=\left(R^i\bar{j}_{\ast}\mathbb{L}\right)_{x} \]
which proves that $\psi\circ\phi$ is an isomorphism. Since $\phi$ is an isomorphism by \cite[10.41]{abbesgros}, so is $\psi$.
\end{proof}

\bibliographystyle{amsplain}
\bibliography{goodnbhdbib}
\end{document}